\documentclass[letterpaper,11pt]{amsart}

\usepackage[T1]{fontenc}
\usepackage[english]{babel}
\usepackage{amsthm}
\usepackage{cases}
\usepackage{amsfonts}
\usepackage{amsmath}
\usepackage{amssymb}

\usepackage{mathtools}
\usepackage{enumerate}
\usepackage{chngcntr}
\usepackage{hyperref}
\counterwithin{equation}{section}
\usepackage{float}
\calclayout

\usepackage{enumitem}

\DeclareMathOperator{\cf}{\mathfrak{cf}}

\DeclareMathOperator{\rank}{rank}

\DeclareMathOperator{\im}{im}

\DeclareMathOperator{\N}{\mathcal{N}}
\DeclareMathOperator{\Sel}{Sel}

\title[The $3$-isogeny Selmer groups]{The $3$-isogeny Selmer groups of the elliptic curves $y^2=x^3+n^2$}
\author{Stephanie Chan}
\date{August 1, 2023}

\thanks{
The author would like to thank Peter Koymans and Carlo Pagano for helpful discussions.}

\address{Department of Mathematics, University of Michigan}
\email{ytchan@umich.edu}

\newcommand{\cF}{\mathcal{F}}

\newcommand{\bu}{\mathbf{u}}
\newcommand{\bv}{\mathbf{v}}
\newcommand{\blam}{\boldsymbol{\lambda}}
\newcommand{\bw}{\mathbf{w}}
\newcommand{\Q}{\mathbb{Q}}

\newcommand{\Z}{\mathbb{Z}}
\newcommand{\F}{\mathbb{F}}
\newcommand{\OO}{\mathcal{O}}
\newcommand{\p}{\mathfrak{p}}

\newcommand{\leg}[2]{\left(\frac{#1}{#2}\right)}

\usepackage{dsfont}

\newcommand{\hphi}{\hat{\phi}}
\newcommand{\hE}{\hat{E}}

\newtheorem{theorem}{Theorem}[section]

\newtheorem{lemma}[theorem]{Lemma}

\newtheorem{definition}[theorem]{Definition}

\begin{document}
\begin{abstract}
Consider the family of elliptic curves $E_n:y^2=x^3+n^2$, where $n$ varies over positive cubefree integers. There is a rational $3$-isogeny $\phi$ from $E_n$ to $\hE_n:y^2=x^3-27n^2$ and a dual isogeny $\hphi:\hE_n\rightarrow E_n$. We show that for almost all $n$, the rank of $\Sel_{\phi}(E_n)$ is $0$, and the rank of $\Sel_{\hphi}(\hE_n)$ is determined by the number of prime factors of $n$ that are congruent to $2\bmod 3$ and the congruence class of $n\bmod 9$.
\end{abstract}

\maketitle

\section{Introduction}
We study a cubic twist family of elliptic curves
\[E_n:y^2=x^3+n^2,\]
where $n$ ranges over cubefree positive integers.
The rational $3$-isogeny from $E_n$ to its dual curve
$\hE_n:y^2=x^3-27n^2$
is given explicitly by
\[\phi:E_n(\Q)\rightarrow \hE_n(\Q)\qquad (x,y)\mapsto\left(\frac{x^3+4n^2}{x^2},\frac{y(x^3-8n^2)}{x^3}\right).\]
The dual isogeny of $\phi$ is
\[\hphi:\hE_n(\Q)\rightarrow E_n(\Q)\qquad
(x,y)\mapsto\left(\frac{x^3-108n^2}{9x^2},\frac{y(x^3+216n^2)}{27x^3}\right).\]
For elliptic curves with rational $3$-isogenies, it is natural to study their $3$-isogeny Selmer groups. 
The $\phi$-Selmer group $\Sel_\phi(E_n)$ is the subgroup of $H^1(G_\Q, E_n[\phi])$ of classes that are locally in the image of the Kummer map 
$\hE_n(\Q_p) \rightarrow H^1(G_{\Q_p}, E_n[\phi])$ for every place $p$ of $\Q$. The $\hphi$-Selmer group $\Sel_{\hphi}(E_n)$ is defined in the same way. We say that a polynomial over $\Q$ is \emph{everywhere locally solvable} if it has a nontrivial zero over $\Q_p$ for every place $p$ of $\Q$. Let $K\coloneqq\Q(\sqrt{-3})$. Following~\cite[Sections~3 and~4]{CohenPazuki}, the Selmer groups can be described more explicitly with the following identifications:
\begin{align}\label{eq:explicitSelhat}
 \Sel_{\hphi}(\hE_n)&=
\{[u]\in \Q^\times/(\Q^{\times})^3: uX^3+u^2Y^3+2nZ^3=0\text{ is everywhere locally solvable}\},\\\label{eq:explicitSel}
\Sel_{\phi}(E_n)&=
\{[u]\in G_3: 
u(X + Y\sqrt{-3})^3 + \overline{u}(X -Y\sqrt{-3} )^3 + 2nZ^3 = 0 \\
& \hspace{20em} \text{ is everywhere locally solvable}\},\notag
\end{align}
where $G_3$ is the subgroup of classes of elements of $K^{\times}/(K^{\times})^3$ whose norm is a cube, and $\overline{u}$ denotes the complex conjugate of $u$. 
For related work on explicit $3$-descent on such curves, see for example~\cite{CohenPazuki,DeLong,Top,FJKRTX}.

To state our results, define the sets
\begin{align*}
 \mathcal{D} &\coloneqq\{D\in\Z: D\geq 1 \text{ cubefree}\},\\
 \mathcal{D}(N) &\coloneqq\{D\in\mathcal{D}:D\leq N\},\\
 \end{align*}
 and for integers $i$ the functions
\begin{align*}
 \omega(n) &\coloneqq \#\{\text{prime }p:p\mid n\},\\
 \omega_i(n) &\coloneqq \#\{\text{prime }p:p\mid n,\ p\equiv i\bmod 3\}.
\end{align*}
Let $\cf(2n)$ denote the cubefree part of $2n$, so for $n\in\mathcal{D}$, we have
\[\cf(2n)=\begin{cases}
2n&\text{if }v_2(n)\neq 2,\\
\frac{1}{4}n&\text{if }v_2(n)= 2,
\end{cases}
\]
where $v_p(n)$ denotes the $p$-adic valuation of an integer $n$.

The groups $\Sel_{\phi}(E_n)$ and $\Sel_{\hphi}(\hE_n)$ are finite $3$-groups, so we make sense to study their ranks $\rank \Sel_{\phi}(E_n)\coloneqq \dim_{\F_3}\Sel_{\phi}(E_n)$ and $\rank \Sel_{\hphi}(\hE_n)\coloneqq \dim_{\F_3}\Sel_{\hphi}(\hE_n)$. 
Our main result is the following theorem, which determines the rank of $\Sel_{\phi}(E_n)$ and $\Sel_{\hphi}(\hE_n)$ for almost all positive cubefree integers $n$ when ordered by the size of $n$.

\begin{theorem}\label{theorem:main}
Let $n\in\mathcal{D}$. Define 
\begin{equation}\label{eq:deltan}
\delta_n =
 \begin{cases}
 \hfil 1&\text{if }n\equiv \pm 3\bmod 9,\\
 -1&\text{if }n\equiv \pm 4 \bmod 9,\\
 \hfil 0&\text{otherwise}.
 \end{cases}
\end{equation}
Then
\begin{equation}\label{eq:tamratio}
 \rank \Sel_{\hphi}(\hE_n)= \rank \Sel_{\phi}(E_n)+\omega_2(\cf(2n))+\delta_n.
\end{equation}
Moreover for any $\epsilon>0$, we have
\begin{equation}\label{eq:selden}
\#\{n\in\mathcal{D}(N):\rank \Sel_{\phi}(E_n)\neq 0\}\ll N(\log N)^{-\frac{1}{3}+\epsilon}.\end{equation}
\end{theorem}
By Cassels' formula~\cite{Casselsformula}, the ratio between $\#\Sel_{\hphi}(\hE_n)$ and $\#\Sel_{\phi}(E_n)$ can be expressed as a product of local Tamagawa ratios. The Tamagawa ratios can be computed explicitly using Tate's algorithm. For the curves $E_n$ and $\hE_n$ here, the computations have been done in~\cite[Proposition~34]{BES}, and~\eqref{eq:tamratio} is immediate by combining all the local factors.

Notice that~\eqref{eq:tamratio} implies a lower bound 
\begin{equation}\label{eq:selhatlb}
 \rank \Sel_{\hphi}(\hE_n)\geq \omega_2(\cf(2n))+\delta_n.
\end{equation}
Furthermore, since $\rank \Sel_{\phi}(E_n)$ is almost always $0$ by~\eqref{eq:selden}, putting this into~\eqref{eq:tamratio} shows that the rank of $\Sel_{\hphi}(\hE_n)$ almost always takes the lower bound in~\eqref{eq:selhatlb}.
In Section~\ref{section:redei}, we show how the lower bound~\eqref{eq:selhatlb} and the relation~\eqref{eq:tamratio} can be visualized more explicitly by studying matrices of cubic residue symbols 
formed from packaging the local solvability conditions of the cubic equations from~\eqref{eq:explicitSelhat} and~\eqref{eq:explicitSel}. We will identify $\Sel_{\hphi}(\hE_n)$ with the kernel of some $\omega_1(n)+r$ by $\omega(n)+c$ matrix, where $r,c\in\{-1,0,1\}$. Then the dimensions of the matrix guarantees the lower bound~\eqref{eq:selhatlb}. Furthermore, as $n$ grows, we expect that usually $\omega_1(n)\approx\frac{1}{2}\log\log n$ and $\omega(n)\approx\log\log n$ in the sense of the Erd\H{o}s–Kac theorem. Heuristically, viewing the entries of the matrix as random elements in $\F_3$ picked uniformly, the probability that the matrix takes the maximum possible rank tends to $1$ as $n$ tends to infinity.
This observation is similar to that in work of Fouvry, Koymans, and Pagano on the $4$-rank of class groups of biquadratic fields~\cite{FKP}. There they show that for almost all odd positive squarefree integers $n$, the $4$-rank of the class group of $\Q(\sqrt{n}, \sqrt{-1})$ is $1$ less than the number of primes dividing $n$ that are congruent to $3\bmod 4$.

To show that the rank of $\Sel_{\hphi}(\hE_n)$ indeed attains this lower bound~\eqref{eq:selhatlb} for almost all $n$, we compute the weighted moments of the size of $\Sel_{\hphi}(\hE_n)$.
This moment computation is analogous to that done by Heath-Brown~\cite{HBSelmer1,HBSelmer} in obtaining the distribution of the $2$-Selmer group of the elliptic curves $y^2=x^3-D^2x$, with $D$ being squarefree. 

\begin{theorem}\label{theorem:phihatSelmer}
For any $\epsilon>0$ and any positive integer $k$, we have
\[\sum_{n\in\mathcal{D}(N)} (\#\Sel_{\phi}(E_n))^k=\sum_{n\in\mathcal{D}(N)} \left(3^{-(\delta_n+\omega_2(\cf(2n)))}\#\Sel_{\hphi}(\hE_n)\right)^k= 
\#\mathcal{D}(N)
+
O_{k,\epsilon}\left(N(\log N)^{-\frac{1}{3}+\epsilon}\right).\]
\end{theorem}

Observe that Theorem~\ref{theorem:phihatSelmer} implies~\eqref{eq:selden} and completes the proof of Theorem~\ref{theorem:main}.
With~\eqref{eq:tamratio}, estimating the moments of $\#\Sel_{\phi}(E_n)$ is the same as estimating the moments of $3^{-(\delta_n+\omega_2(\cf(2n)))}\#\Sel_{\hphi}(\hE_n)$. 
We will prove Theorem~\ref{theorem:phihatSelmer} in Section~\ref{section:phihatSel} by estimating only the weighted moments of $\#\Sel_{\hphi}(\hE_n)$, but we remark that it is also possible to do a similar computation on the moments of $\#\Sel_{\phi}(E_n)$ directly without invoking~\eqref{eq:tamratio}.

In Heath-Brown's work~\cite{HBSelmer1,HBSelmer}, the solvability conditions that determines the elements in the $2$-Selmer group are in terms of quadratic residue symbols. In our case, the $3$-isogeny Selmer elements are controlled by conditions involving cubic residue symbols. To treat the sums, we adapt work of Klys~\cite{Klys} on the distribution of the $\ell$-part of class groups of degree $\ell$ cyclic fields. Klys' work is a generalization of a result by Fouvry--Kl\"{u}ners~\cite{FK4rank} on the distribution of the $4$-part of the class group of quadratic fields, which follows a line of attack similar to work of Heath-Brown~\cite{HBSelmer1,HBSelmer}.

 The estimates in Theorem~\ref{theorem:phihatSelmer} suggest that $\rank \Sel_{\hphi}(\hE_n)- \omega_2(\cf(2n))+\delta_n=\rank \Sel_{\phi}(E_n)$ converges to the Dirac delta distribution. Similar distributions have been observed previously in certain problems related to the $2$-part of class groups~\cite{AK,FKP,KMS}, and for the $2$-Selmer groups of quadratic twists of elliptic curves over a fixed quadratic extensions~\cite{MP}.

It is known that the distribution of $2$-Selmer ranks diverges in certain families of elliptic curves by averaging the lower bound coming from the Tamagawa ratio~\cite{KLO1,KLO2} and by estimating the order of magnitude of the average size of the $2$-Selmer group~\cite{Yu2sel2}.
For elliptic curves with $3$-isogenies, Alp\"{o}ge, Bhargava, and Shnidman~\cite{ABS} proved that the average size of the $2$-Selmer group of $y^2=x^3+dn^2$ is $3$, for any fixed integer $d\neq 0$. By averaging the Tamagawa ratio of the curves with $n\leq N$, they showed that the average size of the $3$-isogeny Selmer group has a lower bound of the order $(\log N)^{\frac{1}{3}}$ when $d$ and $-3d$ are not squares, and order $\log N$ when $-3d$ is a square. When $d$ is a square, they predicted that the average should be bounded.
From Theorem~\ref{theorem:phihatSelmer}, we see that the average size of $\Sel_{\phi}(E_n)$ is $1$.
 As we will see in Section~\ref{section:regularmoments}, with some slight modification to the proof of Theorem~\ref{theorem:phihatSelmer}, it is possible to obtain the moments of $\#\Sel_{\hphi}(\hE_n)$, which in particular shows that the average size of $\Sel_{\hphi}(\hE_n)$ is asymptotic to a constant multiple of $\log N$. We summarize the result as follows.

\begin{theorem}\label{theorem:usualhatmoments}
For any $\epsilon>0$ and any positive integer $k$, we have
\begin{equation}\label{eq:regularmoments}
 \frac{1}{\#\mathcal{D}(N)}\sum_{n\in\mathcal{D}(N)} (\#\Sel_{\hphi}(\hE_n))^k=c_k(\log N)^{\frac{1}{2}(3^k-1)}+O_{k,\epsilon}\left((\log N)^{\frac{3^{k}}{2}-\frac{5}{6}+\epsilon}\right),
\end{equation}
where $c_k$ is a positive constant depending only on $k$ given in \eqref{eq:ck}.
\end{theorem}

The distribution of the $2$-Selmer ranks of quadratic twists of elliptic curves over $\Q$ with full two-torsion converges to some non-constant distribution~\cite{HBSelmer1,HBSelmer,Kane}. In contrast, the size of the $3$-isogeny Selmer groups considered here are almost always fixed after removing the dependencies on $\omega_2(\cf(2n))$ and the congruence class of $n$ modulo $9$. 
Kane and Thorne~\cite{KT} obtained the distribution the $2$-isogeny Selmer groups of the quartic twists $y^2 = x^3-dx$, where $d$ is restricted to a family for which the Tamagawa ratio of the isogeny is fixed. In our setting, the Tamagawa ratio is controlled by $\omega_2(n)$ and $\delta_n$ in~\eqref{eq:tamratio}, and fixing the Tamagawa ratio would likely imply a non-constant distribution as in~\cite[Theorem~1.3]{KT} because of the restrictions on the dimensions of the matrices of cubic residue symbols.

\section{Fundamental \texorpdfstring{$3$}{3}-isogeny descent}
We will take the explicit description of the $3$-isogeny Selmer elements of $E_n$ and $\hE_n$ given by Cohen and Pazuki in~\cite[Sections~3 and~4]{CohenPazuki}.
The fundamental $3$-descent map 
\[\theta:E_n(\Q)\rightarrow \Q^\times/(\Q^{\times})^3\times \Q^\times/(\Q^{\times})^3 \]
is defined by $\theta(O) = (1,1)$, $\theta((0,n)) = ((2n)^2,2n)$, $\theta((0,-n)) = (2n,(2n)^2)$, and 
\[\theta((x,y))=(y-n,y+n)\]
for the non-torsion points of $E_n(\Q)$.

Let $K\coloneqq\Q(\sqrt{-3})$ and $\rho\coloneqq\frac{-1+\sqrt{-3}}{2}$.
The fundamental $3$-descent map on $\hE_n$ is defined as
\[\hat{\theta}:\hE_n(\Q)\rightarrow G_3\times G_3 \qquad
(x,y)\mapsto\left(y-3n\sqrt{-3} ,y+3n\sqrt{-3}\right),\]
where $G_3$ is the subgroup of classes of elements of $K^{\times}/(K^{\times})^3$ whose norm is a cube. 
Note that
$\im(\theta)\cong E_n(\Q)/\hphi(\hE_n(\Q))$ and 
$\im(\hat{\theta})\cong \hE_n(\Q)/\phi(E_n(\Q))$.
Moreover, since the product of the two entries in any tuple $\theta((x,y))$ is always in $\Q^{\times})^3$, the map $\theta$ is defined through only the first coordinate of its output. Similarly, $\hat{\theta}$ is defined through only the first coordinate of its output.

\subsection{Local solvability conditions}
The elements in $\Sel_{\hphi}(\hE_n)$ and $\Sel_{\phi}(E_n)$ can be detected by checking the local solvability of certain homogeneous cubic polynomials at every place of $\Q$. 
We will put the cubics in~\eqref{eq:explicitSelhat} and~\eqref{eq:explicitSel} into more convenient forms.
Define
\[\mathfrak{D}\coloneqq\left\{(A,B,C)\in\Z_{\geq 1}^3: A,B,C \text{ pairwise coprime and cubefree}\right\}.
\]

We first make some small modifications to the equation of the cubics we saw in~\eqref{eq:explicitSelhat} and~\eqref{eq:explicitSel} that were from~\cite{CohenPazuki}.

\begin{lemma}\label{lemma:selrep}
Let $n\in \mathcal{D}$.
Then $\theta_0\in \im\theta$ if and only if there exists some pairwise coprime positive integers $A,B,C$ such that $ABC=\cf(2n)$, $\theta_0=(AB^2,A^2B)$, and 
\begin{equation}\label{eq:cubicABC}
 AX^3+BY^3+CZ^3=0
\end{equation}
has a nontrivial integer solution.
\end{lemma}
\begin{proof}
By~\cite[Theorem 3.1]{CohenPazuki}, $\theta_0\in \im\theta$ if and only if $ab^2X^3+a^2bY^3+2nZ^3=0$ has a nontrivial integer solution, where $a,b$ are squarefree coprime positive integers such that $\theta_0=(ab^2,a^2b)$ and $ab\mid 2n$. 
Write $c=\cf(2n)/ab$, and 
\[(A,B,C)=
\left(\frac{a\gcd(b,c)^2}{\gcd(a,c)},
\frac{b\gcd(a,c)^2}{\gcd(b,c)},
\frac{c}{\gcd(ab,c)}\right)\in\mathfrak{D}.
\]
This substitution gives the cubic in~\eqref{eq:cubicABC}.
\end{proof}

\begin{lemma}
Let $n\in\mathcal{D}$.
Then $\theta_0\in \im\hat{\theta}$ if and only if there exists $\alpha\in\OO_K$ such that $\alpha,\overline{\alpha},\frac{2n}{\alpha\overline{\alpha}}$ are pairwise coprime in $\OO_K$, $\theta_0=(\alpha\overline{\alpha}^2,\alpha^2\overline{\alpha})$, and 
\[
 \alpha(X + Y\sqrt{-3})^3 + \overline{\alpha}(X -Y\sqrt{-3} )^3 + \frac{2n}{\alpha\overline{\alpha}}Z^3 = 0
\]
has a nontrivial integer solution.
\end{lemma}
\begin{proof}
By~\cite[Corollary 4.3]{CohenPazuki},
any $\theta_0\in\im\hat{\theta}$ can be put in the form $\theta_0=(\alpha\overline{\alpha}^2,\alpha^2\overline{\alpha})$, where $\alpha\in\OO_K$ is such that $\alpha\overline{\alpha}\mid 2n$ and $\gcd(\alpha,\overline{\alpha})=1$. Since $2$ is inert in $\OO_K$, we have $\alpha\overline{\alpha}\mid \cf(2n)$.
Furthermore, we can assume that $\alpha\overline{\alpha}$ and $\frac{2n}{\alpha\overline{\alpha}}$ are coprime, otherwise replace $\alpha$ with 
\[\frac{\alpha\gcd(\overline{\alpha},\frac{2n}{\alpha\overline{\alpha}})^2}{\gcd(\alpha,\frac{2n}{\alpha\overline{\alpha}})}.
\]
\end{proof}

Use Lemma~\eqref{lemma:selrep} to rewrite the cubic in \eqref{eq:explicitSelhat}, so 
\[\begin{split}
\#\Sel_{\hphi}(\hE_n)=\#\{(A,B,C)\in\mathfrak{D}: ABC=\cf(2n),\ AX^3+BY^3+CZ^3=0\\ \text{ is locally solvable everywhere}\}.
\end{split}\]
To study the local solvability of 
$AX^3+BY^3+CZ^3=0$,
we summarize below the solvability conditions taken from~\cite[Section~5]{CohenPazuki}.

\begin{lemma}\label{lemma:localcond}
Suppose $(u_1,u_2,u_3)\in\mathfrak{D}$.
Then the equation 
\[u_1X^3+u_2Y^3+u_3Z^3=0\]
has a nontrivial solution in $\Q_p$ if and only if one of the following is satisfied:
\begin{itemize}
\item
$p\neq 3$ and $p\nmid u_1u_2u_3$;
\item
$p\neq 3$, $p\mid u_i$, and $u_j/u_k$ is a cube in $\F_p^\times$, where $\{i,j,k\}=\{1,2,3\}$; 
\item $p=3$, $3\nmid u_1u_2u_3$, and $u_i/u_j\equiv \pm 1\bmod 9$ for some $i\neq j$;
\item $p=3$ and $v_3(u_i) = 1$ for some $i$;
\item $p=3$, $v_3(u_i) = 2$, and $u_j/u_k\equiv \pm 1\bmod 9$, where $\{i,j,k\}=\{1,2,3\}$.
\end{itemize}
\end{lemma}

\subsection{Matrix for \texorpdfstring{$\hphi$}{phihat}-Selmer}\label{section:redei}

In this section, we will interpret $\Sel_{\hphi}(\hE_n)$ and $\Sel_{\phi}(E_n)$ using linear algebra and give a more explicit derivation of the lower bound in~\eqref{eq:selhatlb}.
We construct matrices consisting of cubic residue symbols to study $\Sel_{\hphi}(\hE_n)$. These matrices are reminiscent of the matrices constructed by R\'{e}dei to study the $4$-rank of class groups of quadratic fields~\cite{RedeiMatrix}. For the $3$-isogeny Selmer groups here, some similar construction also appeared in~\cite{FJKRTX}. 

For any irreducible element $\pi$ in $\OO_K$ that does not divide $3$, denote the cubic residue symbol modulo $\pi$ by $\leg{\cdot}{\pi}_3$. For any $\alpha\in \OO_K$ that is coprime to $\pi$, the symbol $\leg{\alpha}{\pi}_3$ takes a value in $\langle \rho\rangle=\{1,\rho,\overline{\rho}\}$ such that
\[\alpha^{\frac{\N(\pi)-1}{3}}\equiv \leg{\alpha}{\pi}_3\bmod \pi.\]
 We extend this multiplicatively to any element $w$ in $\OO_K$ that does not divide $3$, by the formula $\leg{\cdot}{w}_3=\prod_j\leg{\cdot}{\pi_j}_3$, where $w$ is factored into irreducible elements $w=\prod_j \pi_j$. 
See for example \cite[Section~4]{LemReciprocity} for properties of the cubic residue symbol.

Notice that if $[u]\in\Sel_{\hphi}(\hE_n)$, then we also have $[2n/u],[(2n)^2/u]\in\Sel_{\hphi}(\hE_n)$ because of the three-torsion points on $E_n$.
Therefore we can always find a representative $u\in\Z$ that is not divisible by $3$ for any class in $\Sel_{\hphi}(\hE_n)/E_n(\Q)[3]$. By Lemma~\ref{lemma:selrep}, we can also assume $u$ has only prime factors that divides $\cf(2n)$.

Label the distinct prime factors of $\cf(2n)$, except possibly $3$, as $p_1,p_2,\dots, p_{\ell},\dots, p_m$, where $p_1\equiv\dots\equiv p_{\ell}\equiv 1\bmod 3$ and $p_{\ell+1}\equiv\dots\equiv p_{m}\equiv 2\bmod 3$. Take $p_0\coloneqq 3$.
For each $p_i \equiv 1\bmod 3$, fix a choice of irreducible elements $\pi_i\equiv 1\bmod 3$ in $\OO_K$ above $p_i$. For $p_i \equiv 2\bmod 3$, take $\pi_i\coloneqq p_i$. Take $q_i\coloneqq p_i^{v_{p_i}(2n)}$.
Define $r_{ij}\in\F_3$ such that
\[\rho^{r_{ij}}=
\begin{cases}
 \hfil\leg{\rho}{q_j}_3&\text{if }i=0,\\
 \hfil\leg{q_j}{\pi_i}_3 &\text{if }i\neq 0\text{ or }j,\\
\leg{2n/q_i}{\pi_i}^2_3 &\text{if }i=j\neq 0.
\end{cases}\]
Define the matrix 
\[
R=
\begin{cases}
(r_{ij})_{0\leq i\leq \ell,1\leq j\leq m} &\text{if }v_3(n)=0\text{ or }2,\\
(r_{ij})_{1\leq i\leq \ell,1\leq j\leq m} &\text{if }v_3(n)=1.
\end{cases}
\]

Write $\mathbf{e}=(e_1,\dots, e_m)\in \F_3^m$. It is straightforward to check that when $i\neq 0$, the $i$-th entry $(R\mathbf{e})_i$ of $R\mathbf{e}$ satisfies 
\[\rho^{(R\mathbf{e})_i}=\begin{cases}
 \hfil \leg{u}{\pi_i}_3 &\text{if }e_i=0\\
 \leg{(2n)^{e_i}/u}{\pi_i}^{2}_3 &\text{if }e_i\neq 0,
\end{cases}
 \]
 where $u=\prod_{j=1}^m q_j^{e_j}$.
 From the $i=0$ row entry of $R\mathbf{e}$, \[\rho^{(R\mathbf{e})_0}=\leg{\rho}{u}_3.\]
Hence the kernel of the matrix of $R$ contains $\mathbf{e}\in \F_3^m$ such that for every $i\neq 0$, we have
\[\begin{cases}
 \hfil \leg{u}{\pi_i}_3=1 &\text{if }p_i\nmid u,\\
 \hfil \leg{(2n)^{e_i}/u}{\pi_i}_3=1 &\text{if }p_i\mid u,\\
\end{cases}
 \]
 and further that $\leg{\rho}{u}_3=1$ if $v_3(n)=0$ or $2$.
By the first supplementary law of cubic reciprocity~\cite[p.~215]{LemReciprocity}, $\leg{\rho}{u}_3=1$ is equivalent to $u\equiv\pm 1\bmod 9$.
For any $p\equiv 2\bmod 3$, any integer is a cube modulo $p$. Therefore for the corresponding cubic~\eqref{eq:cubicABC} to be everywhere locally solvable, we only need to consider the local conditions at primes dividing $n$ that are congruent to $1\bmod 3$ and at the prime $3$. By \cite[Proposition~4.1(iii)]{LemReciprocity}, to check that an integer $u$ is a cube modulo a prime $p\equiv 1\bmod 3$, it suffices to check that $\leg{u}{\pi}_3=1$ for a choice of irreducible $\pi\in\OO_K$ above $p$.
Therefore the conditions for $\mathbf{e}$ to be in $\ker R$ are precisely the requirements for $u$ to represent a class in $\Sel_{\hphi}(\hE_n)$ by Lemma~\ref{lemma:localcond}.

Suppose $\mathbf{e}\in\ker R$. If $2n\equiv\pm 1\bmod 9$, then $u,2n/u, (2n)^2/u$ all satisfies $\pm 1\bmod 9$. If $2n\equiv\pm 2$ or $\pm 4 \bmod 9$, exactly one of $u,2n/u, (2n)^2/u$ satisfies $\pm 1\bmod 9$.
If $3\mid n$, exactly one of $u,2n/u, (2n)^2/u$ has no valuation at $3$.
Therefore 
\begin{equation}\label{eq:kerR}
 \ker R\cong
\begin{cases}
 \hfil \Sel_{\hphi}(\hE_n) &\text{if }2n\equiv\pm 1\bmod 9,\\
\Sel_{\hphi}(\hE_n)/E_n(\Q)[3]&\text{otherwise.}
\end{cases}
\end{equation}
From the dimension of the matrix $R$, we see that 
\begin{equation}\label{eq:corankR}
\dim\ker R
=m-\rank R
\geq 
\begin{cases}
m-\ell -1&\text{if }v_3(n)=0\text{ or }2,\\
 \hfil m-\ell &\text{if }v_3(n)=1.
\end{cases}
\end{equation}
Combining~\eqref{eq:kerR} and~\eqref{eq:corankR} we get the lower bound for the rank of $\Sel_{\hphi}(\hE_n)$ in~\eqref{eq:selhatlb}.

\subsection{Matrix for \texorpdfstring{$\phi$}{phi}-Selmer}
Now we relate $\Sel_{\phi}(E_n)$ to $\ker R^T$.
The $(i,j)$-entry of $R^T$ is $(R^T)_{ij}=r_{ji}\in\F_3$ such that
\[\rho^{r_{ji}}=
\begin{cases}
 \hfil\leg{\rho}{q_i}_3&\text{if }j=0,\\
 \hfil\leg{q_i}{\pi_j}_3 &\text{if }j\neq 0\text{ or }i,\\
\leg{2n/q_i}{\pi_i}^2_3 &\text{if }i=j\neq 0.
\end{cases}\]
When $j\neq 0$ or $i$, by cubic reciprocity
\[ \leg{p_i}{\pi_j}_3 =\leg{\pi_j}{p_i}_3
=\begin{cases}
\leg{\pi_j}{\pi_i\overline{\pi}_i}_3
= \leg{\pi_j}{\pi_i}_3\leg{\overline{\pi}_j}{\pi_i}_3^2
=\leg{\pi_j\overline{\pi}_j^2}{\pi_i}_3
&\text{if }p_i\equiv 1\bmod 3,\\
\leg{\pi_j}{p_i}_3
=\leg{\pi_j}{p_i}_3^2\leg{\overline{\pi}_j}{p_i}_3
=\leg{\pi_j\overline{\pi}_j^2}{p_i}_3^2 
&\text{if }p_i\equiv 2\bmod 3.
\end{cases}\]
Write $\pi_0\coloneqq\rho$.
We have
\[\rho^{r_{ji}}=
\begin{cases}
 \hfil\leg{\pi_j\overline{\pi}_j^2}{\pi_i}_3^{v_{p_i}(2n)} &\text{if }j\neq i\text{ and } 1\leq i\leq \ell,\\
 \hfil \leg{\pi_j\overline{\pi}_j^2}{p_i}_3^{2v_{p_i}(2n)}&\text{if }\ell+1\leq i\leq m,\\
\leg{2n/\pi_j\overline{\pi}_j^2}{\pi_i}^{2v_{p_i}(2n)}_3 &\text{if }i=j\neq 0.
\end{cases}\] 
Write 
\[
u=
\begin{cases}
\prod_{j=0}^{\ell} (\pi_j\overline{\pi}_j^2)^{e_jv_{p_j}(2n)} &\text{if }v_3(n)=0\text{ or }2,\\
\prod_{j=1}^{\ell} (\pi_j\overline{\pi}_j^2)^{e_jv_{p_i}(2n)} &\text{if }v_3(n)=1,
\end{cases}
\]
and
\[
\mathbf{e}=
\begin{cases}
(e_0,\dots,e_{\ell})\in\F_3^{\ell+1} &\text{if }v_3(n)=0\text{ or }2,\\
(e_1,\dots,e_{\ell})\in\F_3^{\ell} &\text{if }v_3(n)=1.
\end{cases}
\]
The $i$-th entry $(R^T\mathbf{e})_i$ of $R^T\mathbf{e}$ satisfies 
\[\rho^{(R^T\mathbf{e})_i}=\begin{cases}
 \hfil \leg{u}{\pi_i}_3 &\text{if }e_i=0\\
 \leg{(2n)^{e_i}/u}{\pi_i}^{2}_3 &\text{if }e_i\neq 0.
\end{cases}
 \]
Now the kernel of the matrix of $R^T$ can be identified with the elements $u$ such that
\[\begin{cases}
 \hfil \leg{u}{\pi_i}_3=1 &\text{if }\pi_i\nmid u,\\
 \hfil \leg{(2n)^{e_i}/u}{\pi_i}_3=1 &\text{if }\pi_i\mid u,
\end{cases}
 \]
 which are precisely the local conditions for $u$ to represent a class in $\Sel_{\phi}(E_n)$ by~\cite[Section~6]{CohenPazuki}, except possibly at the prime $3$. In fact, the local conditions at $3$ are already encoded in the matrix. 
Summing up the rows of $R^T$, we see that any $\mathbf{e}\in\ker R^T$ satisfies
\[\begin{cases}
\hfil \leg{\rho}{2n}_3^{e_0}=1 &\text{if }v_3(n)=0,\\
\leg{\rho}{2n/3^2}_3^{e_0}\prod_{j=1}^{\ell}\leg{3}{\pi_j}_3^{e_j}=1&\text{if }v_3(n)=2,\\
\hfil \prod_{j=1}^{\ell}\leg{3}{\pi_j}_3^{2e_j}=1&\text{if }v_3(n)=1.
\end{cases}
\]
Applying properties of the cubic residue symbol~\cite[Theorem~7.8]{LemReciprocity}, the equations boil down to the local solvability conditions at the prime $3$ in~\cite[Lemma~6.11]{CohenPazuki}.
 Therefore we conclude that $\ker R^T\cong\Sel_{\phi}(E_n)$.
This allows us to deduce~\eqref{eq:tamratio} from the fact that $\rank R=\rank R^T$.
 
\section{Analytic tools}
This section sets up the analytic results required in the moment computations in Section~\ref{section:phihatSel}.
\subsection{Some classical theorems}
We recall some classical theorems which we will use.
\begin{lemma}[{\cite[Lemma A, p. 265]{HR}}]\label{lemma:HR}
There exists some $B>0$ such that 
\[\#\{1\leq n\leq N:\omega(n)=\ell\}\ll \frac{N}{\log N}\cdot \frac{(\log\log N + B)^{\ell-1}}{(\ell-1) !}\]
uniformly for all positive integers $\ell$ and $N\geq 2$.
\end{lemma}

\begin{lemma}[{\cite[Theorem~1]{Shiu}}]\label{lemma:Shiu}
Fix $0<\epsilon<1$ and some positive constant $A$.
Let $f$ be a multiplicative function such that $f(p^{\ell})\leq A$ for every prime $p$ and $\ell\geq 1$.
Then
\[\sum_{N-M<n\leq N}f(n)\ll_{A,\epsilon}
\frac{M}{\log N}\exp\left(\sum_{p\leq N}\frac{f(p)}{p}\right)
\]
uniformly for
$2\leq N^{\epsilon}\leq M< N$.
\end{lemma}

We will need a version of the Siegel--Walfisz theorem for character sums.
\begin{lemma}[{\cite{Goldstein}}]\label{lemma:SW}
Let $\epsilon>0$. Let $\chi$ be a nontrivial finite Hecke character of $K=\Q(\sqrt{-3})$ with conductor $\mathfrak{f}$. There exists a positive constant $c_{\epsilon}$ such that
\[
\sum_{\substack{\p\subset\OO_K\text{ prime}\\ \N(\p)\leq N\\\p\nmid \mathfrak{f}}}
\chi(\p)\ll_{\epsilon}
|\N(\mathfrak{f})|^{\epsilon} N(\log N)^2
\exp\left(-\frac{c_{\epsilon}(\log N)^{\frac{1}{2}}}{|\N(\mathfrak{f})|^{\epsilon}}\right),
\]
where $\N$ denotes the norm relative to the extension $K/\Q$.
\end{lemma}

\subsection{A large sieve}
Recall that $K=\Q(\sqrt{-3})$ and we denote the norm relative to the extension $K/\Q$ by $\N$.
We prove a version of~\cite[Theorem~2]{HBcubic} that allows us to bound sums of the cubic residue symbols over cubefree but not necessarily squarefree elements in $\OO_K$. We first prove a lemma we need about the oscillations of the cubic residue symbols.
\begin{lemma}\label{lemma:embedding}
Suppose $\alpha\in \OO_K$, $3\nmid\N(\alpha)$ and $\alpha\neq 0$,
then
\[
\sum_{\xi\in (\Z/\N(\alpha)\Z)^\times}\left(\frac{\xi}{\alpha}\right)_3 
=
\begin{cases}
\#(\Z/\N(\alpha)\Z)^\times& \text{if }\alpha\in\langle\rho\rangle\cdot\Z\cdot (\OO_K)^3,\\
 \hfil0&\text{otherwise.}
\end{cases}
\]
\end{lemma}
\begin{proof}
We can write $\alpha=m\cdot \beta$, where $m$ is an integer and $\beta$ is a product of split primes $(\pi)$ in $\OO_K$, such that none of $(\overline{\pi})$ divides $\beta$.
Since $\xi\in (\Z/\N(\alpha)\Z)^\times$ and $m$ are coprime integers in $\Z$, $\leg{\xi}{m}_3=1$ by a property of the cubic residue symbol (see for example~\cite[Proposition~9.3.4]{IrelandRosen}).
Then 
\[\leg{\xi}{\alpha}_3=\leg{\xi}{\beta}_3.\]
Moreover $(\Z/\N(\alpha)\Z)^\times\rightarrow (\Z/\N(\beta)\Z)^\times$ is surjective because $\N(\beta)\mid \N(\alpha)$.

Now consider the natural map
$(\Z/\N(\beta)\Z)^\times\rightarrow (\OO_K/\beta\OO_K)^{\times}$
that takes any integer representative to its class modulo $\beta\OO_K$. We will show that this map is bijective.
If $a$ and $b$ are both integers and $a\equiv b\bmod \beta\OO_K$, then we also have $a\equiv b\bmod \overline{\beta}\OO_K$, so $\N(\beta)\mid a-b$.
It follows that the map must be injective. Since $\#(\Z/\N(\beta)\Z)^\times=\# (\OO_K/\beta\OO_K)^{\times}=\N(\beta)\prod_{p\mid \N(\beta)}(1-p^{-1})$, the map is a bijection.
Therefore as long as $(\beta)$ is not a cube of an ideal in $\OO_K$, we have
\[
\sum_{\xi\in (\Z/\N(\beta)\Z)^\times}\left(\frac{\xi}{\beta}\right)_3 
=\sum_{\xi\in (\OO_K/\beta\OO_K)^{\times}}\left(\frac{\xi}{\beta}\right)_3
= 0.
\]
If $\beta\in(\OO_K)^3$, then clearly $\left(\frac{\xi}{\beta}\right)_3=1$ for every $\xi\in (\Z/\N(\beta)\Z)^\times$. 
This completes the proof.
\end{proof}

 We are now ready to prove the large sieve result we need. For simplicity of the proof, the exponents in Lemma~\ref{lemma:largesieve} are not optimal, but still sufficient for our application.
\begin{lemma}\label{lemma:largesieve}
Let $(\alpha_w)_w$ and $(\beta_z)_z$ be two sequences of complex numbers with $|\alpha_w|\leq 1$ and $|\beta_z|\leq 1$. Then there exists an absolute constant $C>0$ such that
\begin{equation}\label{eq:BMN}
B_1(M, N) \coloneqq \sum_{\substack{w\in\OO_K\, \text{cubefree}\\ \N(w)\leq M \\ w\equiv 1\bmod 3}}\ \sum_{\substack{z\in\OO_K\, \text{cubefree}\\ \N(z)\leq N \\ z\equiv 1\bmod 3}}\alpha_w\beta_z\left(\frac{z}{w}\right)_3 
\ll
(M^{-\frac{1}{16}}+N^{-\frac{1}{16}})
MN(\log MN)^C.
\end{equation}If instead $w$ ranges over cubefree positive integers while $z$ still ranges over cubefree elements of $\OO_K$, we have
\begin{equation}\label{eq:BMNint}
B_2(M, N) \coloneqq \sum_{\substack{w\in\Z\, \text{cubefree}\\1\leq w\leq M }}\ \sum_{\substack{z\in\OO_K\, \text{cubefree}\\\N(z)\leq N \\ z\equiv 1\bmod 3}}\alpha_w\beta_z\left(\frac{z}{w}\right)_3 
\ll
(M^{-\frac{1}{16}}+N^{-\frac{1}{28}})
MN(\log MN)^C.
\end{equation}
\end{lemma}
\begin{proof}
We follow the argument in \cite[Proposition~4.3]{KR}, which is a generalisation of \cite[Section~21]{FI1} and \cite[Proposition~3.6]{KM16}. 
To show~\eqref{eq:BMN}, by cubic reciprocity we can assume that $N\geq M$.
Fix an integer $k \geq 6$, and apply H\"{o}lder's inequality to the variable $w$ with $\frac{k-1}{k}+\frac{1}{k}=1$ to get
\[
|B_1(M, N)|^{k} \ll \left(\sum_w|\alpha_w|^{\frac{k}{k-1}}\right)^{k-1}\sum_{w}\left|\sum_{z}\beta_z\left(\frac{z}{w}\right)_3\right|^{k}
.\]
We drop the condition that $\N(z)\leq N$ and $z$ being cubefree from the sum by adapting the convention that $\beta_z=0$ whenever $\N(z)>N$ or $z$ is not cubefree.
Expand the inner sum in the second factor,
\begin{equation}\label{eq:BMN1}
|B_1(M, N)|^{k} \ll M^{k-1}\sum_{w}\sigma_w\sum_{\N(z)\leq N^k}\beta_z'\left(\frac{z}{w}\right)_3,
\end{equation}
where
\[\sigma_w=\left(\frac{\left|\sum_{z}\beta_z\left(\frac{z}{w}\right)_3\right|}{\sum_{z}\beta_z\left(\frac{z}{w}\right)_3}\right)^k \quad\text{ and }\quad
\beta_z' = \sum_{\substack{z_1\cdots z_{k}=z \\ z_1\equiv\cdots\equiv z_{k}\equiv 1\bmod 3 }}\beta_{z_1}\beta_{z_2}\cdots \beta_{z_{k}}.
\]

Noting that $z$ has at most $2\omega(\N(z))$ prime factors, and $\beta_{z_1}\dots\beta_{z_k}$ is only nonzero when $z_1,\dots,z_k$ are all cubefree, there are $\leq (3^k)^{2\omega(\N(z))}$ ways to write $z$ as a product of $k$ cubefree ideals, so $|\beta_z'|\leq 9^{k\omega(\N(z))}$.

Applying the Cauchy--Schwarz inequality to the $z$ variable in~\eqref{eq:BMN1}, we get
 \begin{equation}\label{eq:CS2}
\left|\sum_{z}\beta_z'\sum_{w}\sigma_w\left(\frac{z}{w}\right)_3\right|^2
 \leq
\left(\sum_{z} \left|\beta_z'\right|^2\right)\left(\sum_{z}\left|\sum_{w}\sigma_w\left(\frac{z}{w}\right)_3\right|^2\right)
\end{equation}

Since a product of $k$ cubefree ideals must be $(2^k+1)$-power free,
we can bound the first sum in~\eqref{eq:CS2} by
\begin{align*}
\sum_{z} \left|\beta_z'\right|^2
&\leq \sum_{n\leq N^k}\#\{z\in\OO_K:\N(z)=n,\ z\text{ is }(2^k+1)\text{-powerfree}\}\cdot 9^{2k\omega(n)}\\
&\leq \sum_{n\leq N^k}
((2^k+1)9^{2k})^{\omega(n)}\ll N^k(\log N)^{C_k},
\end{align*}
where $C_k$ is a positive integer depending only on $k$, and the final inequality is by applying Lemma~\ref{lemma:Shiu}.
Putting this back to~\eqref{eq:CS2} and expanding the inner sum, we have
\[
\begin{split}
\left|\sum_{z}\beta_z'\sum_{w}\sigma_w\left(\frac{z}{w}\right)_3\right|^2
& \ll
N^k(\log N)^{C_k}\sum_{z}
\left(
\sum_{w_1}\sigma_{w_1}\left(\frac{z}{w_1}\right)_3
\sum_{w_2}\overline{\sigma_{w_2}}\overline{\left(\frac{z}{w_2}\right)_3}\right)\\
& \ll
N^k(\log N)^{C_k}
\sum_{w_1}\sum_{w_2}\sigma_{w_1}\overline{\sigma_{w_2}}\sum_{z}\left(\frac{z}{w_1w_2^2}\right)_3.
\end{split}\]
Break up the sum over $z$ into congruence classes $\xi$ modulo $w_1w_2^2$. 
We have
\begin{equation}\label{eq:xiclass}
\sum_{\xi\bmod w_1w_2^2}\left(\frac{\xi}{w_1w_2^2}\right)_3 = 0
\end{equation}
unless $w_1w_2^2\in\langle\rho\rangle\cdot (\OO_K)^3$. This gives
\[
\sum_{\N(z)\leq N^k}\left(\frac{z}{w_1w_2^2}\right)_3\ll
\begin{cases} 
 \hfil N^{k} & \text{if }w_1w_2^2\in\langle\rho\rangle\cdot (\OO_K)^3,\\
M^3N^{\frac{k}{2}} + M^6 & \text{otherwise.}
\end{cases}
\]
Since we took $k\geq 6$ and since $N\geq M$, we have $N^{\frac{k}{2}}\geq M^3$, so the last bound can be simplified to $M^3N^{\frac{k}{2}}$. Hence~\eqref{eq:BMN1} becomes
\[\begin{split}
|B_1(M, N)|^{k} 
&\ll
M^{k-1}(N^k(\log N)^{C_k}
(MN^k+
M^2M^3N^{\frac{k}{2}}))^{\frac{1}{2}}\\
&\ll
(\log N)^{C_k}
(M^{k-\frac{1}{2}}N^{k}+
M^{k+\frac{3}{2}}N^{\frac{3k}{4}})
.
\end{split}\]
Finally, taking $k=8$ we get $|B_1(M, N)| 
\ll
M^{-\frac{1}{16}}
MN(\log N)^{C_8}$. This proves~\eqref{eq:BMN}.

We now turn to the proof of~\eqref{eq:BMNint}.
Assuming $N\geq M$, the proof for $B_2(M,N)$ is the same as the above proof for $B_1(M,N)$.
Now consider the case when $N<M$, we swap the variables $w$ and $z$, and instead look at 
\[B_2(N, M) = \sum_{\substack{\N(w)\leq N \\ w\equiv 1\bmod 3}}
\sum_{z\leq M }
\alpha_w\beta_z\left(\frac{z}{w}\right)_3.\]
The proof proceed as before, but apply Lemma~\ref{lemma:embedding} in place of~\eqref{eq:xiclass} to get
\[
\sum_{\xi\bmod \N(w_1w_2^2)}\left(\frac{\xi}{w_1w_2^2}\right)_3 = 0,
\]
 unless $w_1w_2^2\in\langle\rho\rangle\cdot \Z\cdot (\OO_K)^3$. When $w_1w_2^2\in\langle\rho\rangle\cdot \Z\cdot (\OO_K)^3$ holds, $\N(w_1w_2^2)$ is a product of a square and a cube, there are $\ll M^{\frac{3}{2}}$ many distinct $\N(w_1w_2^2)\leq M^3$. For each value of $n=\N(w_1w_2^2)$, there are $\leq 2\omega(n)$ distinct prime factors of $n$ in $\OO_K$. Since $w_1$ and $w_2$ are cubefree, and $w_1\equiv w_2\equiv 1\bmod 3$ fixes a unique generator for each of the ideals $(w_1)$ and $(w_2)$, there are $\leq 3^{2\omega(n)}$ many possible $w_1$ and $w_2$ given $n$.
Therefore
\begin{align*}
&\#\{(w_1,w_2)\in\OO_K^2: w_1w_2^2\in\langle\rho\rangle\cdot \Z\cdot (\OO_K)^3,\ \N(w_1w_2^2)\leq M^3,\ w_1,w_2\text{ cubefree},\ w_1\equiv w_2\equiv 1\bmod 3\}\\
&\leq \sum_{n\leq M^{\frac{3}{2}}}3^{2\omega(n)}\ll M^{\frac{3}{2}}(\log M)^{C_k},\end{align*}
where we have applied again Lemma~\ref{lemma:Shiu}.
Evaluating the sum $B_2(N,M)$ using this estimate for $\sum_{w_1}\sum_{w_2} 1$, we get
\[\begin{split}
|B_2(N, M)|^{k} 
&\ll
M^{k-1}\left(N^k(\log N)^{C_k}
(M^{\frac{3}{2}}(\log N)^{C_k}N^k+
M^2M^3N^{\frac{k}{2}})\right)^{\frac{1}{2}}\\
&\ll
(\log N)^{C_k}
(M^{k-\frac{1}{4}}N^{k}+
M^{k+\frac{3}{2}}N^{\frac{3k}{4}})
.
\end{split}\]
Take instead $k=7$ in this case. This gives $B_2(N, M) \ll M^{-\frac{1}{28}}MN(\log MN)^{C_7}$.
\end{proof}

\section{Moments of the size of the weighted \texorpdfstring{$\hphi$}{phihat}-Selmer-Selmer}\label{section:phihatSel}
The goal of this section is to prove Theorem~\ref{theorem:phihatSelmer}.
Define
\begin{align*}
 \mathcal{D}_\eta &\coloneqq\{D\in\mathcal{D}:D\equiv \eta \bmod 9\},\\
\mathcal{D}_\eta (N)&\coloneqq\{D\in\mathcal{D}_\eta :D\leq N\}.
\end{align*}

Our first task is to set up an expression for $(\#\Sel_{\hphi}(\hE_n))^k$, which we can handle using the machinery by Heath-Brown \cite{HBSelmer1,HBSelmer}.
We put together the solvability conditions for the cubic~\eqref{eq:cubicABC} from Lemma~\ref{lemma:localcond}.
Fix an integer $\eta\in\{0,\pm1,\pm2,\pm3,\pm4\}$.
Suppose $n\in \mathcal{D}_\eta$.
For each prime $p\mid n$ that is congruent to $1\bmod 3$, we fix an irreducible element $\pi\in\OO_K$ above $p$ that satisfies $\pi\equiv 1\bmod 3$. 
If we take an integer $m\in\Z$ that is coprime to $\pi$, observe that
\[1+\leg{m}{\pi}_3+\leg{m}{\pi}_3^2=
\begin{cases}
3&\text{if }m\text{ is a cube}\bmod \pi,\\
0&\text{otherwise}.
\end{cases}\]
With this observation, we can now package the required local conditions.
For any $(A,B,C)\in\mathfrak{D}$ such that $3\nmid BC$ and $ABC=\cf(2n)$, define
\[\begin{split}
g(A,B,C)&\coloneqq
\frac{1}{3^{\varepsilon_\eta }}\left(1+\leg{\rho}{B/C}_3+\leg{\rho}{B/C}_3^2\right)^{\varepsilon_\eta }
\prod_{\substack{p\mid A\\p\equiv 1\bmod 3}}
\frac{1}{3}
\left(1+\leg{B/C}{\pi^{v_p(A)}}_3
+\leg{B/C}{\pi^{v_p(A)}}_3^2\right)\\
&\prod_{\substack{p\mid B\\p\equiv 1\bmod 3}}
\frac{1}{3}
\left(1+\leg{C/A}{\pi^{v_p(B)}}_3
+\leg{C/A}{\pi^{v_p(B)}}_3^2\right)
\prod_{\substack{p\mid C\\p\equiv 1\bmod 3}}
\frac{1}{3}
\left(1+\leg{A/B}{\pi^{v_p(C)}}_3
+\leg{A/B}{\pi^{v_p(C)}}_3^2\right)\\
&=
\begin{cases}
1 &
\parbox[t]{.8\textwidth}{if $AX^3+BY^3+CZ^3=0$ has nontrivial local solutions \\ everywhere and ($B\equiv\pm C\bmod 9$ if $\varepsilon_\eta =1$),}\\
0 &\text{otherwise,}
\end{cases}
\end{split}
\]
where
\[\varepsilon_\eta \coloneqq
 \begin{cases}
 0&\text{if }\eta\equiv \pm 3 \bmod 9,\\
 1&\text{otherwise}.
 \end{cases}
 \]
We can view $g(A,B,C)$ as a product of indicator functions, each checks the local solubility condition at a prime dividing $3ABC$.
 
We wish to sum $g(A,B,C)$ over all factorization of $\cf(2n)$ such that $3\nmid BC$.
We first address the undercounting caused by the condition $3\nmid BC$ and ($B\equiv\pm C\bmod 9$ if $\varepsilon_\eta =1$), which is tracked by $g(A,B,C)$ but not a required solvability condition of the cubic by Lemma~\ref{lemma:localcond}.
When $2n\equiv \pm 1\bmod 9$, the condition $B/C\equiv \pm 1\bmod 9$ is satisfied by all triples in $\Sel_{\hphi}(\hE_n)$.
If $2n\equiv 0,\pm 2,\pm 3,\pm 4\bmod 9$, assuming $3\nmid BC$ and $B/C\equiv \pm 1\bmod 9$ means that $\sum g(A,B,C)$ is only counting $1/3$ of $\#\Sel_{\phi}(E_n)$. Therefore we need an extra factor of $3$ when $\eta\equiv 0,\pm 1,\pm 2,\pm 3\bmod 9$.
This allows us to write
\[\#\Sel_{\hphi}(\hE_n)
=3^{\delta_\eta +\varepsilon_\eta }\sum_{\substack{(A,B,C)\in\mathfrak{D}\\
ABC=\cf(2n)\\ 3\nmid BC}}g(A,B,C),
\]
where $\delta_\eta$ is as defined in~\eqref{eq:deltan}.

Now expand the product $g(A,B,C)$ as a sum. The product over $p\mid A$ can be rewritten as
\[\prod_{\substack{p\mid A\\p\equiv 1\bmod 3}}
\frac{1}{3}
\left(1+\leg{B/C}{\pi^{v_p(A)}}_3
+\leg{B/C}{\pi^{v_p(A)}}_3^2\right)
=\frac{1}{3^{\omega_1(A)}}\sum_{\substack{D_{00},D_{01},D_{02}\text{ pairwise coprime}\\
p\mid D_{01}D_{02}\Rightarrow p\equiv 1\bmod 3
\\
D_{00}D_{01}D_{02}=A}}
\leg{B/C}{D_{01}}_3
\leg{B/C}{D_{02}}_3^2.
\]
The products over $p\mid B$ and $p\mid C$ can be expanded similarly.
With the notations $A=D_{00}D_{01}D_{02}$, $B=D_{10}D_{11}D_{12}$, $C=D_{20}D_{21}D_{22}$ for the variables in the expansion, we get
\begin{equation}\label{eq:hsel1}
\#\Sel_{\hphi}(\hE_n)=
3^{\delta_\eta -\omega_1(n)}
\sum_{\lambda\in\F_3^{\varepsilon_\eta }}
\sum_{(D_{\bu})}
\prod_\bu \leg{\rho}{D_{\bu}}_3^{R_1(\lambda,\bu)}
\prod_{\bu,\bv}
\leg{D_{\bu}}{d^{(1)}_\bv}_3^{\Phi_1(\bu,\bv)}
,
\end{equation}
where $\Phi_1(\bu,\bv)=v_2(u_1-v_1)$, $R_1(\lambda,\bu)=\lambda u_1$, and the sum is taken over all $D_{\bu}$, $\bu=(u_1,u_2)\in\F_3^2$ such that 
\begin{itemize}
\item $\cf(2n)=\prod_{\bu\in\F_3^2}D_{\bu}$, where $D_{\bu}$ are pairwise coprime positive integers;
\item $d^{(1)}_{\bu}=\prod_{p\equiv 1\bmod 3}\pi^{v_p(D_{\bu})}\in \OO_K$;
\item $D_\bu$ is not divisible by any prime $p\equiv 1\bmod 3$ if $u_2\neq 0$;
\item $3\nmid D_{\bu}$ if $\bu\neq\mathbf{0}$.
\end{itemize}

Next we raise~\eqref{eq:hsel1} to the $k$th-power as in~\cite[Section~2]{HBSelmer} and~\cite[Section~5]{FK4rank}. To carry this out, we need to systematically index the variables to deal with the $k$-fold expansion.
\begin{definition}[Indices]
We call any element in $\F_3^{2k}$ an index.
For any index 
\[\bu =(u^{(1)}_1,\dots,u^{(k)}_1,u^{(1)}_2,\dots,u^{(k)}_2)\in \F_3^{2k},\]
write $\bu^{(j)}\coloneqq(u^{(j)}_1,u^{(j)}_2)$,
$\bu_1\coloneqq \pi_1(\bu)\coloneqq(u^{(1)}_1,\dots,u^{(k)}_1)$, and $\bu_2\coloneqq\pi_2(\bu)\coloneqq(u^{(1)}_2,\dots,u^{(k)}_2)$.
For 
$\blam =(\lambda_1,\dots,\lambda_k)\in \F_3^{k}$ and any indices 
$\bu,\bv \in \F_3^{2k}$,
define the functions
\[\Phi_k(\bu,\bv)=\bv_2\cdot (\bu_1-\bv_1) =\sum_{j=1}^{k}v^{(j)}_{2}(u^{(j)}_{1}-v^{(j)}_{1}),\qquad R_k(\blam,\bu)=\blam\cdot \bu_1=\sum_{j=1}^{k}\lambda_j u^{(j)}_{1}.
\]
\end{definition}
With this indexing system defined, using $\bu,\bv$ to denote indices in $\F_3^{2k}$, we get the expansion
 \[
(\#\Sel_{\hphi}(\hE_n))^k=
3^{k(\delta_\eta -\omega_1(n))}
\sum_{\blam\in(\F_3^k)^{\varepsilon_\eta }}
\sum_{(D_\bu)}
\prod_\bu
\leg{\rho}{D_{\bu}}_3^{R_k(\blam ,\bu)}
\prod_{\bu,\bv}
\leg{D_{\bu}}{d_\bv^{(1)}}_3^{\Phi_k(\bu,\bv)}
,
\]
where $(D_{\bu})$ varies over factorizations that satisfies $\cf(2n)=\prod_{\bu\in\F_3^{2k}}D_{\bu}$.

It is apparent that~\eqref{eq:hsel1} does not depend on the choice of $\pi$ in the expression, so we may average over the two possible choices of irreducible $\pi, \overline{\pi}\equiv 1\bmod 3$ in $\OO_K$ above each rational prime $p\equiv 1\bmod 3$ that divides $n$.
The gives the factor $2^{-\omega_1(n)}$ in Lemma~\ref{lemma:Selkexp} below.

\begin{lemma}\label{lemma:Selkexp}
Let $n\in\mathcal{D}_\eta $.
We have 
\begin{equation}\label{eq:Selkexp}
(\#\Sel_{\hphi}(\hE_n))^k=
3^{k\delta_\eta }
(3^{k}\cdot 2)^{-\omega_1(n)}
\sum_{\blam\in(\F_3^k)^{\varepsilon_\eta }}
\sum_{(d_\bu)}
\prod_\bu
\leg{\rho}{D_{\bu}}_3^{R_k(\blam ,\bu)}
\prod_{\bu,\bv}
\leg{D_{\bu}}{d_\bv^{(1)}}_3^{\Phi_k(\bu,\bv)}
,
\end{equation}
where the sum is taken over all
$(d_{\bu})=(d_{\bu}^{(1)}d_{\bu}^{(2)})_{\bu\in\F_3^{2k}}$ such that 
\begin{itemize}
\item $d^{(1)}_{\bu}\in\OO_K$ is a product of irreducibles $\pi\equiv 1\bmod 3$ in $\OO_K$ above rational primes $p\equiv 1\bmod 3$,
\item $d^{(2)}_{\bu}\in \Z$ is a product of rational primes $p\not\equiv 1\bmod 3$,
\item $\cf(2n)=\prod_{\bu\in\F_3^{2k}}D_{\bu}$, where $D_{\bu}=d^{(1)}_{\bu}\overline{d^{(1)}_{\bu}}d^{(2)}_{\bu}$ are pairwise coprime positive integers;
\item $d^{(2)}_\bu=1$ if $\bu_2\neq \mathbf{0}$;
\item $3\nmid D_{\bu}$ if $\bu\neq\mathbf{0}$.
\end{itemize}
\end{lemma}

If we pair the two terms involving a given pair of $\bu$ and $\bv$ in the inner product of~\eqref{eq:Selkexp} and apply cubic reciprocity, we get 
\begin{multline}\label{eq:expandproduct}
\leg{D_{\bu}}{d_\bv^{(1)}}_3^{\Phi_k(\bu,\bv)}\leg{D_{\bv}}{d_\bu^{(1)}}_3^{\Phi_k(\bv,\bu)}\\
=\leg{d^{(1)}_{\bu}}{d_\bv^{(1)}}_3^{\Phi_k(\bu,\bv)+\Phi_k(\bv,\bu)}\leg{\overline{d^{(1)}_{\bu}}}{d_\bv^{(1)}}_3^{\Phi_k(\bu,\bv)-\Phi_k(\bv,\bu)}
\leg{d^{(2)}_{\bu}}{d_\bv^{(1)}}_3^{\Phi_k(\bu,\bv)}
\leg{d^{(2)}_{\bv}}{d_\bu^{(1)}}_3^{\Phi_k(\bv,\bu)}.
\end{multline}
Observe that the right hand side of~\eqref{eq:expandproduct} will be constantly $1$ whenever $\Phi_k(\bu,\bv)=\Phi_k(\bv,\bu)=0$. If we take a set of indices such that this property holds pairwise, then the product of symbols $\leg{D_{\bu}}{d_\bv^{(1)}}_3^{\Phi_k(\bu,\bv)}$ over this set would be constantly $1$. This motivates the following definition.
\begin{definition}[Unlinked indices]
We call two indices $\bu,\bv\in\F_3^{2k}$ unlinked if \[\Phi_k(\bu,\bv)=\Phi_k(\bv,\bu)=0.\]
Otherwise, we say that they are linked.
\end{definition}

To proof Theorem~\ref{theorem:phihatSelmer}, our goal is to estimate
\begin{equation}\label{eq:defSeta}
S_\eta (N,k)
\coloneqq 
\sum_{ n\in\mathcal{D}_\eta (N)}3^{-k(\delta_\eta+\omega_2(\cf(2n)))}(\#\Sel_{\hphi}(\hE_n))^k
\end{equation}
for each $\eta\in \{0,\pm 1,\pm 2, \pm 3, \pm 4\}$.
Note that 
\begin{equation}\label{eq:relateSeta}
\sum_{n\in\mathcal{D}(N)} \left(3^{-(\delta_n+\omega_2(\cf(2n)))}\#\Sel_{\hphi}(\hE_n)\right)^k= \sum_{\eta=-4}^{4}S_\eta (N,k).\end{equation}
Put~\eqref{eq:Selkexp} into the expression~\eqref{eq:defSeta}, we have
\[
S_\eta (N,k)
=
\sum_{ n\in\mathcal{D}_\eta (N)}
3^{-k\omega_2(\cf(2n))}\cdot
(3^k\cdot 2)^{-\omega_1(n)}
\sum_{\blam\in(\F_3^k)^{\varepsilon_\eta }}
\sum_{(d_\bu)}
\prod_\bu
\leg{\rho}{D_{\bu}}_3^{R_k(\blam ,\bu)}
\prod_{\bu,\bv}
\leg{D_{\bu}}{d_\bv^{(1)}}_3^{\Phi_k(\bu,\bv)}
.\]

Write $D^{(1)}_{\bu}=d^{(1)}_{\bu}\overline{d^{(1)}_{\bu}}$ and $D^{(2)}_{\bu}=d^{(2)}_{\bu}$, so $D_{\bu}=D^{(1)}_{\bu}D^{(2)}_{\bu}$ . Dissect the sum according to the size of each $D_\bu^{(j)}$ with the dissection parameter 
\[\Delta\coloneqq 1+(\log N)^{-3^k},\] and according to the divisibility of $D_\bu$ by $2$.
For each $i\in \{1,2\}$ and $\bu\in\F_3^{2k}$, define a number $A^{(i)}_{\bu}$ of the form $1,\Delta,\Delta^2,\dots$. 
Also define \[\Omega\coloneqq e3^k(\log\log N+B),\]
where $B$ is the constant in Lemma~\ref{lemma:HR}.

For each $\mathbf{A}=(A^{(i)}_{\bu})_{i,\bu}$ and $\nu\in\{0,1,2\}$,
define the restricted sum
\begin{equation}\label{eq:restrictedsum}
S_{\eta ,\nu}(N,k,\mathbf{A})
\coloneqq 
\sum_{\blam\in(\F_3^k)^{\varepsilon_\eta }}
\sum_{(d_\bu)}
\prod_\bu 3^{-k\omega_2(D_{\bu})} (3^k\cdot 2)^{-\omega_1(D_{\bu})}
\leg{\rho}{D_{\bu}}_3^{ R_k(\blam,\bu)}
\prod_{\bu,\bv}
\leg{D_{\bu}}{d^{(1)}_\bv}_3^{\Phi_k(\bu,\bv)},
\end{equation}
where the sum is taken over all $(d_\bu^{(i)})_{i,\bu}$ such that
\begin{align*}
 &\prod_\bu D_\bu=\cf(2n)\text{ for some }n\in\mathcal{D}_\eta (N),\ v_2(\cf(2n))=\nu,\\
 & A^{(i)}_\bu\leq D^{(i)}_\bu<\Delta A^{(i)}_\bu,\ \omega(D^{(i)}_\bu)< \Omega.
\end{align*}
Notice that from Lemma~\ref{lemma:Selkexp}, $D^{(2)}_\bu\neq 1$ only if $\bu_2=\mathbf{0}$. Therefore we may assume
$A^{(2)}_\bu= 1$ if $\bu_2\neq\mathbf{0}$.

For ease of notation, we henceforth suppress any $k$ and $\epsilon$ in the big-$O$ and $\ll$ notations, since our estimates are allowed to depend on $k$ and $\epsilon$.
We also define here two numbers
\begin{align*}
 N^{\dagger}&\coloneqq(\log N)^{28(1+C+(9^k+3^k)(1+3^k))},\\
N^{\ddagger}&\coloneqq\exp((\log N)^{3^{-k}\epsilon}),
\end{align*}
where $C>0$ is the constant in Lemma~\ref{lemma:largesieve}.
We will later use these to split into cases according to the size of some $A_{\bu}^{(i)}$ relative to $N^{\dagger}$ and $N^{\ddagger}$.

\subsection{Number of prime factors of the variables}
To write $S_\eta (N,k)$ in terms of $S_{\eta ,\nu}(N,k,\mathbf{A})$, we bound the contribution from those $(D_\bu)$ such that
$\omega(D_\bu)\geq \Omega$ holds for some $\bu\in\F_3^{2k}$.

\begin{lemma}\label{lemma:largeomega}
\[S_\eta (N,k)=
\sum_{\nu=0}^2\sum_{\mathbf{A}} S_{\eta ,\nu}(N,k,\mathbf{A})+O(N(\log N)^{-1}),
 \]
 where the sum is over $\mathbf{A}$ is such that \[\prod_{i=1}^2\prod_{\bu\in\F_3^{2k}} A^{(i)}_\bu \leq 
 \begin{cases} \frac{1}{4}N&\text{if } \nu=0,\\
2N&\text{if } \nu\in\{1,2\}.
 \end{cases}
 \] 
\end{lemma}
\begin{proof}
If $\omega(D_\bu)\geq \Omega$ for some $\bu\in\F_3^{2k}$, then $( D_\bu)$ must come from some $n\in\mathcal{D}(N)$ with $\omega(n)\geq \Omega$. 
Suppose $p\mid n$. If $p\equiv 1 \bmod 3$, there are $3^{2k}$ ways to place $p$ in one of the $D^{(1)}_{\bu}$ and there are two choices of irreducibles $\pi$ above $p$ that satisfies $\pi\equiv 1\bmod 3$. If $p\not\equiv 1\bmod 3$, there are $3^{k}$ ways to place $p$ in one of the $D^{(2)}_{\bu}$, since $D_{\bu}^{(2)}\neq 1$ only if $\bu_2=\mathbf{0}$.
Therefore the contribution from such $n$ to $S_\eta (N,k)$ is bounded by
\[\ll \sum_{\substack{n\leq N\\ \omega(n)\geq \Omega}}9^{k\omega_1(n)}3^{k\omega_2(n)}3^{-k\omega(n)}
=
\sum_{\substack{n\leq N\\ \omega(n)\geq \Omega}}3^{k\omega_1(n)}
\leq 
\sum_{\substack{n\leq N\\ \omega(n)\geq \Omega}}3^{k\omega(n)}.
\]
Now by Lemma~\ref{lemma:HR},
\begin{equation}\label{eq:poie}
\ll
\frac{N}{\log N}
\sum_{\ell\geq \Omega}
\frac{3^{k\ell}(\log\log N+B)^{\ell}}{\ell!}
= 
\frac{N}{\log N}
\sum_{\ell\geq \Omega}
\frac{(3^{k}\log\log N+3^kB)^{\ell}}{\ell!}.
\end{equation}
Bound the sum with a geometric series
\[
\sum_{\ell\geq \Omega}
\frac{(3^{k}\log\log N+3^kB)^{\ell}}{\ell!}
\leq
\frac{(3^{k}\log\log N+3^kB)^{\lceil\Omega\rceil}}{\lceil\Omega\rceil!}
\sum_{m\geq 0}
\frac{(3^{k}\log\log N+3^kB)^{m}}{\Omega^m}.
\]
The last sum converges because $\Omega>3^k(\log\log N+B)$.
Using Stirling's approximation $n!>\sqrt{2\pi n}(n/e)^{n}$, we get a lower bound for $\lceil\Omega\rceil!$. Then
\[
\frac{(3^{k}\log\log N+3^kB)^{\lceil\Omega\rceil}}{\lceil\Omega\rceil!}\ll \frac{1}{\Omega^{\frac{1}{2}}}.
\]
Putting this back to~\eqref{eq:poie}, we have the bound
$\ll N(\log N)^{-1}$.
\end{proof}

\subsection{Incomplete boxes}
Having obtained Lemma~\ref{lemma:largeomega}, we want to estimate $S_{\eta ,\nu}(N,k,\mathbf{A})$ for different $\mathbf{A}$. We start by bounding those that will land into the error term of Theorem~\ref{theorem:phihatSelmer}. First is the contribution from those $\mathbf{A}$ that define boxes not entirely contained in $\mathcal{D}(N)$.
We bound the contribution from 
\[\cF_1\coloneqq\left\{\mathbf{A}:\prod_i\prod_\bu A^{(i)}_\bu\geq \Delta^{-(9^k+3^k)}N\right\}.\]
If $n\in\mathcal{D}(N)$ has a factorization that is in some boxes defined by $\mathbf{A}\in\cF_1$, then clearly $n\geq \Delta^{-(9^k+3^k)}$. We bound the terms from such $n$ in a similar manner as in the proof of Lemma~\ref{lemma:largeomega}.
We see that
\[\sum_{\mathbf{A}\in\cF_1} |S_{\eta ,\nu}(N,k,\mathbf{A})|\ll
 \sum_{\Delta^{-(9^k+3^k)}N\leq n\leq N}
9^{k\omega_1(n)}3^{-k\omega_1(n)}
\ll \left(1-\Delta^{-(9^k+3^k)}\right)N(\log N)^{\frac{1}{2}(3^k-1)},
\]
where we applied Lemma~\ref{lemma:Shiu} in the last inequality.
Since $\Delta=1+(\log N)^{-3^k}$,
we have
\begin{equation}\label{eq:incomplete}
\sum_{\mathbf{A}\in\cF_1} |S_{\eta ,\nu}(N,k,\mathbf{A})|\ll N(\log N)^{-1}.\end{equation}

\subsection{Few large indices}
The second set of $\mathbf{A}$ are those with very few large indices. We will bound the contribution from the set
\[\begin{split}
 \cF_2\coloneqq &
 \left\{\mathbf{A}:\#\{\bu\in\F_3^{2k}:A^{(1)}_{\bu}\geq N^{\ddagger}\}+ \#\{\bu\in\F_3^{2k}:A^{(2)}_{\bu}\geq N^{\ddagger}\}\leq 3^k+3^{k-1}\right\}.
\end{split}\]
We partition $\cF_2$ by defining
\[\cF_2(r,s)\coloneqq \left\{\mathbf{A}:\#\{\bu\in\F_3^{2k}:A^{(1)}_{\bu}\geq N^{\ddagger}\}=r,\ \#\{\bu\in\F_3^{2k}:A^{(2)}_{\bu}\geq N^{\ddagger}\}=s\right\}.\]

Let $n$ be the product of those $D^{(i)}_\bu>N^{\ddagger}$ and $m$ be the product of the remaining ones, then
\[
\sum_{\mathbf{A}\in\cF_2(r,s)} |S_{\eta ,\nu}(N,k,\mathbf{A})|
\ll
\sum_{m\leq (N^{\ddagger})^{9^k-r}}
(9^k-r)^{\omega_1(m)}(3^k-s)^{\omega_2(m)}
3^{-k\omega(m)}
\sum_{n\leq N/m}r^{\omega_1(n)}s^{\omega_2(n)}
3^{-k\omega(n)}.
\]
Apply Lemma~\ref{lemma:Shiu} to the last sum, then by Mertens theorem 
\[
\sum_{\mathbf{A}\in\cF_2(r,s)} |S_{\eta ,\nu}(N,k,\mathbf{A})|
\ll
N
(\log N)^{\frac{1}{2}\left(\frac{r+s}{3^k}\right)-1}
\sum_{m\leq (N^{\ddagger})^{9^k}}
\frac{3^{k\omega(m)}}{m}
\ll 
N(\log N)^{\frac{1}{2}\left(\frac{r+s}{3^k}\right)-1}
(\log N^{\ddagger})^{3^{k}}
. 
\]
Summing over $(r,s)$ in $\{r+ s\leq 3^k+ 3^{k-1}\}$, we have
\begin{equation}\label{eq:fewlarge}
\sum_{\mathbf{A}\in\cF_2} |S_{\eta ,\nu}(N,k,\mathbf{A})|\ll
\sum_{r,s}
N(\log N)^{\frac{1}{2}\left(\frac{r+s}{3^k}\right)-1}
(\log N^{\ddagger})^{3^{k}}
\ll
N(\log N)^{-\frac{1}{3}+\epsilon}.
\end{equation}

\subsection{Two large linked indices}
Looking at the expression of $S_{\eta,\nu}(N,k,\mathbf{A})$ in~\eqref{eq:restrictedsum}, if there is a pair of linked indices $\bu,\bv$, we expect that the product of the symbols involving $\bu,\bv$ would oscillate in the box defined by $\mathbf{A}$ unless $A_{\bu}=1$ or $A_{\bv}=1$. Here we treat the case when there are two linked indices that are large using the large sieve in Lemma~\ref{lemma:largesieve}.
Let
\[\cF_3\coloneqq\{\mathbf{A}: A^{(i)}_\bu,\ A^{(1)}_\bv\geq N^{\dagger}\text{ for some }i \text{ and linked } \bu,\bv\}\setminus \cF_1.\]
If $\Phi_k(\bu,\bv)$ and $\Phi_k(\bv,\bu)$ are both nonzero and $i=1$, collecting the symbols involving both $\bu$ and $\bv$,~\eqref{eq:expandproduct} becomes one of
\[\overline{\leg{d^{(1)}_\bu}{d^{(1)}_\bv}_3},\ \leg{\overline{d^{(1)}_\bu}}{d^{(1)}_\bv}_3,\ 
\leg{d^{(1)}_\bu}{\overline{d^{(1)}_\bv}}_3,\
\leg{d^{(1)}_\bu}{d^{(1)}_\bv}_3.
\]
Otherwise, that is, if exactly one of $\Phi_k(\bu,\bv)$ and $\Phi_k(\bv,\bu)$ is nonzero, or $i=2$, only one of $\leg{D^{(i)}_\bu}{d^{(1)}_\bv}_3$ and $\leg{D^{(1)}_\bv}{d^{(i)}_\bu}_3$ can appear in the sum.

For any $\mathbf{A}\in\cF_3$, rewrite the sum in terms the symbols involving both of the linked indices $\bu,\bv$, we have
\[\begin{split}
& |S_{\eta,\nu}(N,k,\mathbf{A})|\\
&\ll
\sum_{\blam\in(\F_3)^{\varepsilon_\eta}}
\sum_{\substack{(d^{(r)}_\bw)\\ (r,\bw)\neq (i,\bu), (1,\bv)}}
\left|\sum_{d^{(i)}_\bu}\sum_{d^{(1)}_\bv}
a_{\blam}(d^{(i)}_\bu,(d^{(r)}_\bw))a_{\blam}(d^{(1)}_\bv,(d^{(r)}_\bw))
\leg{D^{(i)}_{\bu}}{d^{(1)}_\bv}_3^{\Phi_k(\bu,\bv)}
\leg{D^{(1)}_{\bv}}{d^{(1)}_\bu}_3^{(2-i)\Phi_k(\bv,\bu)}\right|,
\end{split}
\]
where
\[
a_{\blam}(d^{(i)}_\bu,(d^{(r)}_\bw))=
3^{-k\omega(D_{\bu})}2^{-\omega_1(D_{\bu})}
\leg{\rho}{D_{\bu}}_3^{ R_k(\blam ,\bu)}
\prod_{\bw\neq\bu,\bv}
\leg{D^{(i)}_{\bu}}{d^{(1)}_\bw}_3^{\Phi_k(\bu,\bw)}
\prod_{\substack{(r,\bw)\neq (1,\bv)\\ \bw\neq \bu}}
\leg{D^{(r)}_{\bw}}{d^{(1)}_\bu}_3^{(2-i)\Phi_k(\bw,\bu)}
.
\]
Apply Lemma~\ref{lemma:embedding}, then since
$A^{(i)}_\bu,A^{(1)}_\bv\geq N^{\dagger}$,
we obtain
\[|S_{\eta,\nu}(N,k,\mathbf{A})|\ll N(N^{\dagger})^{-\frac{1}{28}}(\log N)^C\]
for some absolute constant $C$.
Sum over the $O((\log N)^{(9^k+3^k)(3^k+1)})$ possible $\mathbf{A}$ and plug in the definition of $N^{\dagger}$, then
\begin{equation}\label{eq:largelinked}
 \sum_{\mathbf{A}\in\cF_3}|S_{\eta,\nu}(N,k,\mathbf{A})|\ll N(N^{\dagger})^{-\frac{1}{28}}(\log N)^{C+(9^k+3^k)(3^k+1)}
\ll N(\log N)^{-1}.
\end{equation}

\subsection{A large and a small linked indices}
Given a pair of linked indices, if one of the variables in the pair is small and nontrivial, we can use the Siegel--Walfisz theorem in Lemma~\ref{lemma:SW} to obtain cancellation when the other variable varies in some large range. 
More precisely, define
\[\cF_4\coloneqq\{\mathbf{A}: 2\leq A^{(i)}_\bv<N^{\dagger},\ A^{(1)}_\bu\geq N^{\ddagger}\text{ for some }i\text{ and some linked }\bu,\bv\}\setminus (\cF_1\cup\cF_3).\]

Suppose $\mathbf{A}\in\cF_4$ and fix $\bu$ to be such that $A^{(1)}_\bu\geq N^{\ddagger}$. Let $\mathcal{V}\subseteq \F_{3}^{2k}$ be the set of $\bv$ that are linked with $\bu$ and satisfy $A^{(1)}_\bv \geq 2$ or $A^{(2)}_\bv\geq 2$. We assume that $\bu$ is chosen such that $\mathcal{V}$ is non-empty. Since $\mathbf{A}\notin\cF_3$, any $\bv\in\mathcal{V}$ must satisfy $A^{(i)}_\bv<N^{\dagger}$.

For any given $(D_{\bv})_{\bv\in\mathcal{V}}$ that is in the range defined by $\mathbf{A}$, that is, $A^{(i)}_\bv\leq D^{(i)}_\bv<\Delta A^{(i)}_\bv$, set
\[\chi (d)=\leg{\rho}{\N(d)}_3^{ R_k(\blam,\bu)}
\prod_{\bv\in\mathcal{V}}\leg{\N(d)}{d^{(1)}_\bv}_3^{\Phi_k(\bu,\bv)}
\leg{D_\bv}{d}_3^{\Phi_k(\bv,\bu)}\]
for $d\in\OO_K$.
The character $\chi$ is nontrivial by construction because any $\bv\in\mathcal{V}$ is linked with $\bu$ and satisfies $D^{(i)}_\bv\geq A^{(i)}_\bv\geq 2$ for at least one of $i\in\{1,2\}$.
Plugging in $d=d^{(1)}_\bu$, we can rewrite the expression as
\[\chi (d^{(1)}_\bu)=\leg{\rho}{D^{(1)}_\bu}_3^{ R_k(\blam,\bu)}
\prod_{\bv\in\mathcal{V}}\leg{D^{(1)}_\bu}{d^{(1)}_\bv}_3^{\Phi_k(\bu,\bv)}
\leg{D_\bv}{d^{(1)}_\bu}_3^{\Phi_k(\bv,\bu)}.\]
Then
\[S_{\eta,\nu}(N,k,\mathbf{A})\ll
\max_{\blam\in(\F_3)^k}
\sum_{(D_\bv)_{\bv\neq \bu}, D_\bu^{(2)}}
\left|\sum_{d^{(1)}_\bu}
(3^{k}\cdot 2)^{-\omega_1(D_{\bu})}
\chi(d^{(1)}_\bu)\right|.
\]
Now split the sum according to the number of prime factors of $D_{\bu}^{(1)}$, and apply the assumption $\omega(D_{\bu}^{(1)})<\Omega$ from the definition of~\eqref{eq:restrictedsum},
\[
S_{\eta,\nu}(N,k,\mathbf{A})\ll 
\max_{\blam\in(\F_3)^k}
\sum_{(D_\bv)_{\bv\neq \bu},D^{(2)}_\bu}
\sum_{\ell=1}^{\lfloor\Omega\rfloor}
\frac{1}{(3^k\cdot 2)^{\ell}}
\left|
\sum_{\substack{d^{(1)}_\bu\\ \omega(D^{(1)}_u)=\ell}}
\chi(d^{(1)}_\bu)
\right|.
\]

Write $D^{(1)}_\bu=mp_{\ell}^{r_{\ell}}=p_1^{r_1}\cdots p_{\ell}^{r_{\ell}}$, where $r_i\in\{1,2\}$ and $p_1<p_2<\dots<p_\ell$. Denote by $\pi_\ell$ any irreducible element in $\OO_K$ above $p_\ell$ that satisfies $\pi_\ell\equiv 1\bmod 3$.
Since $\chi$ is multiplicative, we can take out the factor $\chi(\pi_{\ell})^{r_{\ell}}$, so 
\begin{equation}\label{eq:preSW}
S_{\eta,\nu}(N,k,\mathbf{A})\ll 
\max_{\blam\in(\F_3)^k}
\sum_{(D_\bv)_{\bv\neq \bu},D^{(2)}_\bu}
\sum_{\ell=1}^{\lfloor\Omega\rfloor}
\sum_{\substack{m\leq \Delta A_{\bu}^{(1)}\\ \omega(m)=\ell-1}}
\sum_{r_\ell=1}^{2}
\left|
\sum_{\substack{\pi_\ell\equiv 1\bmod 3\\ \N(\pi_\ell)\text{ prime}}} \chi(\pi_{\ell})
\right|,
\end{equation}
where the innermost sum is over
\begin{equation}\label{eq:pinormrange}
\max\left\{\left(\frac{A^{(1)}_\bu}{m}\right)^{\frac{1}{r_{\ell}}},\ p_{\ell-1}\right\}
<\N(\pi_{\ell})=p_{\ell}
<\left(\frac{\Delta A^{(1)}_\bu}{m}\right)^{\frac{1}{r_{\ell}}}
.\end{equation}

Our next task is to apply Lemma~\ref{lemma:SW} to the innermost sum of~\eqref{eq:preSW}.
From~\eqref{eq:pinormrange}, we deduce that 
\[
\left(\frac{\Delta A^{(1)}_\bu}{m}\right)^{\frac{1}{r_{\ell}}}>\N(\pi_{\ell})=p_{\ell}>( A^{(1)}_\bu)^{\frac{1}{r_1+\cdots+r_{\ell}}}>(N^{\ddagger})^{\frac{1}{2\Omega}}
.\]
Since $\chi$ has modulus dividing $9\prod_{\bv\in\mathcal{V}} D_{\bv}\ll (N^{\dagger})^{2\#\mathcal{V}}$, and trivially $\#\mathcal{V}\leq 2\cdot 9^k$, we can bound the norm of the conductor of $\chi$ by
$\N(\mathfrak{f})\ll(N^{\dagger})^{4\cdot 9^k}$.
For each prime ideal in $\OO_K$ that splits in $K/\Q$, there is exactly one generator that is congruent to $1\bmod 3$. 
The statement in Lemma~\ref{lemma:SW} includes the contribution from inert primes, but this can be removed because there are only $\ll\sqrt{x}$ primes in $\OO_K$ with norm $\leq x$ that are inert in $K/\Q$. 
Hence Lemma~\ref{lemma:SW} implies that, for any $\tau>0$, we have
\[\sum_{\substack{\pi_\ell=1\bmod 3\\ \N(\pi_\ell)\text{ satisfies }\eqref{eq:pinormrange}\\ \N(\pi_\ell)\text{ prime}}}\chi(\pi_{\ell})
\ll _{\tau}
(N^{\dagger})^{\tau}\left(\frac{ A^{(1)}_\bu}{m}\right)^{\frac{1}{r_{\ell}}}
\exp\left(-\frac{c}{(N^{\dagger})^{\tau}}
\left(\frac{\log N^{\ddagger}}{\Omega}\right)^{\frac{1}{2}}\right)+ \Omega
\]
for some constant $c>0$ depending on $\tau$, and where the term $\Omega$ comes from those $\pi_\ell$ that may divide some $D_\bv$.
Putting this back into~\eqref{eq:preSW} and sum over all integers $m\leq \Delta A_{\bu}^{(1)}/(N^{\ddagger})^{\frac{1}{2\Omega}}$, we get
\[S_{\eta,\nu}(N,k,\mathbf{A})\ll_{\tau}
\sum_{(D_\bv)_{\bv\neq \bu},D^{(2)}_\bu}
(N^{\dagger})^{\tau}
A^{(1)}_\bu\exp\left(-\frac{c}{(N^{\dagger})^{\tau}}
\left(\frac{\log N^{\ddagger}}{\Omega}\right)^{\frac{1}{2}}\right)
.\]
Then summing over $(D_{\bv})_{\bv\neq\bu}$, and $D_\bu^{(2)}$, we arrive at
\[S_{\eta,\nu}(N,k,\mathbf{A})
\ll_{\tau} N(N^{\dagger})^{\tau}
\exp\left(-\frac{c}{(N^{\dagger})^{\tau}}\left(\frac{\log N^{\ddagger}}{\Omega}\right)^{\frac{1}{2}}\right).
\]
Finally pick $\tau$ small enough relative to $k$ and $\epsilon$, and sum over the $O((\log N)^{((9^k+3^k)(3^k+1)})$ possible $\mathbf{A}$, we have
\begin{equation}\label{eq:largesmalllinked}
 \sum_{\mathbf{A}\in\cF_4}|S_{\eta,\nu}(N,k,\mathbf{A})|
\ll N(\log N)^{-1}.
\end{equation}

\subsection{Geometry of unlinked indices}
In the sum $S_{\eta,\nu}(N,k,\mathbf{A})$, when there is a pair of linked indices $\bu,\bv$, we should be able to use~\eqref{eq:largelinked}~or~\eqref{eq:largesmalllinked} to obtain enough cancellations from the oscillations of the symbols, unless $A_{\bu}=1$ or $A_{\bv}=1$, which forces $D_{\bu}=1$ or $D_{\bv}=1$. Therefore it make sense to study sets of indices that are pairwise unlinked in order to obtain the main term.
\begin{definition}[Maximal unlinked set]
We say that a set of indices is unlinked if any two indices in the set are unlinked. We call an unlinked set maximal if it is not contained in any other unlinked set.
\end{definition}
We will show that the main term of Theorem~\ref{theorem:phihatSelmer} comes from the terms $S_{\eta,\nu}(N,k,\mathbf{A})$ when $A^{(i)}_{\bu}= 1$ for all $\bu$ outside of a maximal set of unlinked indices.
The shape of maximal unlinked sets with respect to the function $\Phi_k$ is available from work of Klys~\cite[Section~8]{Klys}, which we summarize below.

\begin{lemma}[{\cite[Section~8]{Klys}}]\label{Lemma:unlinkedform}
The maximal unlinked sets are of the form
\[\mathbf{a}+V,\]
where $\mathbf{a}\in\F_3^{2k}$ and $V$ is a subspace of $\F_3^{2k}$ of dimension $k$. Moreover $\pi_2(\mathbf{a})\in \pi_1(V)^{\perp}$ and $V=\pi_1(V)\oplus\pi_1(V)^{\perp}$.
\end{lemma}

We now gather those $S_{\eta,\nu}(N,k,\mathbf{A})$ that we have not bounded previously in~\eqref{eq:incomplete},~\eqref{eq:fewlarge},~\eqref{eq:largelinked} and~\eqref{eq:largesmalllinked}.
\begin{lemma}
If $\mathbf{A}\notin\cF_1\cup\cF_2\cup\cF_3\cup\cF_4$, then we have
\[A^{(1)}_\bv=A^{(2)}_\bv=1\text{ for all } \bv\notin\F_3^{k}\times\{\mathbf{0}\}\subseteq \F_3^{2k} .\]
\end{lemma}
\begin{proof}
Take $\mathcal{U}=\{\bu:A^{(1)}_{\bu}\geq N^{\ddagger}\}$ and $\mathcal{V}=\{\bu:A^{(2)}_{\bu}\geq N^{\ddagger}\}$. By assumption $\mathbf{A}\notin\cF_3$, so $\mathcal{U}$ is unlinked, hence $\#\mathcal{U}\leq 3^{k}$. Recall that $A^{(2)}_\bu= 1$ for all $\pi_2(\bu)\neq\mathbf{0}$, so $\#\mathcal{V}\leq 3^{k}$.
Since $\mathbf{A}\notin\cF_2$, we have $\#\mathcal{U}+\#\mathcal{V}> 3^{k}+ 3^{k-1}$, so $\#\mathcal{U}> 3^{k-1}$ and $\#\mathcal{V}> 3^{k-1}$. The maximal unlinked sets are cosets of vector spaces of dimension $k$ by Lemma~\ref{Lemma:unlinkedform}, so any two maximal unlinked sets must have intersection of size at most $3^{k-1}$. This implies that $\mathcal{U}$ is contained in exactly one maximal unlinked set $\mathcal{M}$.
Moreover $\mathcal{V}$ must be in a maximal unlinked set containing $\mathcal{U}$ since~$\mathbf{A}\notin\cF_3$, so $\mathcal{V}\subseteq\mathcal{M}$.

Since $\#\mathcal{V}> 3^{k-1}$ and $A^{(2)}_\bu= 1$ for all $\pi_2(\bu)\neq\mathbf{0}$, the only possibility is that $\mathcal{M}= \F_3^{k}\times\{\mathbf{0}\}$.
If there exists some $\bv\notin\mathcal{U}$ that is unlinked with every $\bu\in\mathcal{U}$, then $\mathcal{U}\cup\{\bv\}$ is also an unlinked set, so $\bv\in\F_3^{k}\times\{\mathbf{0}\}$.
Therefore any $\bv\notin\F_3^{k}\times\{\mathbf{0}\}$ is linked to some $\bu\in\mathcal{U}$.
Then $\mathbf{A}\notin\cF_3\cup\cF_4$ implies that $A^{(1)}_{\bv}=A^{(2)}_{\bv}=1$ for all $\bv\notin \F_3^k\times\{\mathbf{0}\}$.
\end{proof}

The leading term of $S_\eta(N,k)$ therefore comes from $\mathcal{M}=\F_3^k\times\{\mathbf{0}\}$.
Define
\begin{equation}\label{eq:SetaM}
S_\eta^{\mathcal{M}}(N,k)\coloneqq \sum_{\nu=0}^2\sum_{\substack{\mathbf{A}\\ \bv\notin\mathcal{M}\Rightarrow A^{(i)}_\bv=1}}S_{\eta,\nu}(N,k,\mathbf{A}).
\end{equation}
Put together the error terms from~\eqref{eq:incomplete},~\eqref{eq:fewlarge},~\eqref{eq:largelinked},~\eqref{eq:largesmalllinked} and Lemma~\ref{lemma:largeomega}.
\begin{lemma}\label{lemma:maximalerror}
Let $\mathcal{M}=\F_3^k\times\{\mathbf{0}\}$. Then
\[S_\eta(N,k)=S_\eta^{\mathcal{M}}(N,k)+O(N(\log N)^{-\frac{1}{3}+\epsilon}).
\]
\end{lemma}

We now compute the leading term for Theorem~\ref{theorem:phihatSelmer}. Put $A^{(i)}_\bv=1$ for all $\bv\notin\mathcal{M}$ into~\eqref{eq:restrictedsum}.
\[S_\eta^{\mathcal{M}}(N,k)= 
\sum_{n\in\mathcal{D}_\eta(N)}
3^{-k(\omega_1(\cf(2n))+\omega_2(\cf(2n)))}
\sum_{\blam\in(\F_3^k)^{\varepsilon_\eta}}
\sum_{(D_\bu)}
\prod_{\bu\in\mathcal{M}}\leg{\rho}{D_\bu}_3^{ R_k(\blam,\bu)}
.
\]
Expanding the inner sum we get
\[S_\eta^{\mathcal{M}}(N,k)= 
\sum_{\blam\in(\F_3^k)^{\varepsilon_\eta}}
\sum_{n\in\mathcal{D}_\eta(N)}
3^{-k(\omega_1(\cf(2n))+\omega_2(\cf(2n)))}
\prod_{\substack{p\mid \cf(2n)\\p\neq 3}} 
\sum_{\bu\in\mathcal{M}}\leg{\rho}{p^{v_p(2n)}}_3^{ R_k(\blam,\bu)}
.\]
When $p\neq 3$ and $v_p(2n)\neq 0$, we can check that
\[\sum_{\bu\in\mathcal{M}}\leg{\rho}{p^{v_p(2n)}}_3^{ R_k(\blam,\bu)}
=
\begin{cases}
3^k&\text{if $\blam=\mathbf{0}$ or $p\equiv \pm 1\bmod 9$,}\\
 \hfil 0&\text{otherwise}.
\end{cases}
\]
Take the term indexed by $\blam=\mathbf{0}$ as the main term and bound the rest by Lemma~\ref{lemma:Shiu}, so
\begin{align*}
S_\eta^{\mathcal{M}}(N,k)
-\#\mathcal{D}_\eta(N)
&=2\varepsilon_{\eta}\sum_{\substack{n\in\mathcal{D}_\eta(N)\\ p\mid n\Rightarrow p=3\text{ or}\\p\equiv \pm1\bmod 9}}
3^{-k(\omega_1(\cf(2n))+\omega_2(\cf(2n)))}
\prod_{\substack{p\mid n\\ p\equiv \pm1\bmod 9}}
3^k\\
&\ll\sum_{\substack{n\leq N\\ p\mid n\Rightarrow p=3\text{ or}\\p\equiv \pm1\bmod 9}}
1
\ll N(\log N)^{-\frac{2}{3}}
.\end{align*}
Combining 
\[S_\eta^{\mathcal{M}}(N,k)
=
\#\mathcal{D}_\eta(N)
+
O\left(N(\log N)^{-\frac{2}{3}}\right)
\]
and Lemma~\ref{lemma:maximalerror},
we conclude that
\[S_\eta(N,k)=
\#\mathcal{D}_\eta(N)
+
O\left(N(\log N)^{-\frac{1}{3}+\epsilon}\right).
\]
Putting this back to the expression in~\eqref{eq:relateSeta} completes the proof of Theorem~\ref{theorem:phihatSelmer}.

\section{Moments of the size of \texorpdfstring{$\hphi$}{phihat}-Selmer}\label{section:regularmoments}
The asymptotics of the moments of $ \#\Sel_{\hphi}(\hE_n)$ can be obtained by modifying the proof of Theorem~\ref{theorem:phihatSelmer}.

\begin{theorem}\label{theorem:omega2moments}
For any $\epsilon>0$ and any positive integer $k$, we have
\begin{equation}\label{eq:omega2moments}
 \sum_{n\in\mathcal{D}(N)} (\#\Sel_{\hphi}(\hE_n))^k
=
 \sum_{n\in\mathcal{D}(N)} 3^{k(\delta_n+\omega_2(\cf(2n)))}\cdot \#\mathcal{D}(N)
+O_{k,\epsilon}\left(N(\log N)^{\frac{3^{k}}{2}-\frac{5}{6}+\epsilon}\right).
\end{equation}\end{theorem}

\begin{proof}
To compute $\sum_{n\in\mathcal{D}(N)} (\#\Sel_{\hphi}(\hE_n))^k$, it suffices to repeat the proof of Theorem~\ref{theorem:phihatSelmer} with a weight of $3^{k(\delta_\eta+\omega_2(\cf(2n)))}$.
We accordingly define weighted version of~\eqref{eq:defSeta} and~\eqref{eq:restrictedsum} with an extra factor of $3^{k(\delta_\eta+\omega_2(\cf(2n)))}$, 
\begin{align*}
S_\eta' (N,k)
&\coloneqq 
\sum_{ n\in\mathcal{D}_\eta (N)}(\#\Sel_{\phi}(E_n))^k
,\\
S_{\eta ,\nu}'(N,k,\mathbf{A})
&\coloneqq 3^{k\delta_\eta}
\sum_{\blam\in(\F_3^k)^{\varepsilon_\eta }}
\sum_{(d_\bu)}
\prod_\bu 3^{-k\omega_1(D_{\bu})} 2^{-\omega_1(D_{\bu})}
\leg{\rho}{D_{\bu}}_3^{ R_k(\blam,\bu)}
\prod_{\bu,\bv}
\leg{D_{\bu}}{d^{(1)}_\bv}_3^{\Phi_k(\bu,\bv)}.
\end{align*}
It is routine to check that with this extra weight, the estimates in Lemma~\ref{lemma:largeomega},~\eqref{eq:incomplete},~\eqref{eq:largelinked} and~\eqref{eq:largesmalllinked} still hold.
However, the upper bound in~\eqref{eq:fewlarge} will need to be replaced.
Placing this weight and follow the same argument as before, we have
\begin{align*}
\sum_{\mathbf{A}\in\cF_2(r,s)} |S_{\eta ,\nu}'(N,k,\mathbf{A})|
&\ll
\sum_{m\leq (N^{\ddagger})^{9^k-r}}
(9^k-r)^{\omega_1(m)}(3^k-s)^{\omega_2(m)}
3^{-k\omega_1(m)}
\sum_{n\leq N/m}r^{\omega_1(n)}s^{\omega_2(n)}
3^{-k\omega_1(n)}\\
&\ll 
N(\log N)^{\frac{1}{2}\left(\frac{r}{3^k}+s\right)-1}
(\log N^{\ddagger})^{3^{k}}
. 
\end{align*}
The next step is to sum over $(r,s)$ in $\{r+ s\leq 3^k+ 3^{k-1}\}$. Since that $A^{(2)}_\bu= 1$ for all $\bu_2\neq\mathbf{0}$, so $s\leq 3^{k}$. Therefore $\frac{r}{3^k}+s$ is maximum in this set when $s=3^k$ and $r=3^{k-1}$.
Therefore we have
\[
\sum_{\mathbf{A}\in\cF_2} |S_{\eta ,\nu}'(N,k,\mathbf{A})|\ll
\sum_{r,s}
N(\log N)^{\frac{1}{2}\left(\frac{r}{3^k}+s\right)-1}
(\log N^{\ddagger})^{3^{k}}
\ll
N(\log N)^{\frac{3^k}{2}-\frac{5}{6}+\epsilon}.
\]

Define a weighted version of~\eqref{eq:SetaM},
\[{S'}_\eta^{\mathcal{M}}(N,k)\coloneqq \sum_{\nu=0}^2\sum_{\substack{\mathbf{A}\\ \bv\notin\mathcal{M}\Rightarrow A^{(i)}_\bv=1}}S_{\eta,\nu}'(N,k,\mathbf{A}).\]
By the same argument as before, the corresponding statement for Lemma~\ref{lemma:maximalerror} becomes
\[S'_\eta(N,k)={S'}_\eta^{\mathcal{M}}(N,k)+O\left(N(\log N)^{\frac{3^k}{2}-\frac{5}{6}+\epsilon}\right)
\]
and we also have the estimate \[
{S'}_\eta^{\mathcal{M}}(N,k)
-\sum_{n\in\mathcal{D}_\eta(N)}
3^{k(\delta_\eta+\omega_2(\cf(2n))}
\ll\sum_{\substack{n\leq N\\ p\mid n\Rightarrow p=3\text{ or}\\p\equiv \pm1\bmod 9}}
3^{k\cdot \#\{p\mid n\: :\: p\equiv -1\bmod 9\}}
\ll N(\log N)^{\frac{3^{k-1}}{2}-\frac{5}{6}}
\]
by an application of Lemma~\ref{lemma:Shiu}.
Putting everything together we get~\eqref{eq:omega2moments} as claimed.
\end{proof}

To understand the growth of~\eqref{theorem:omega2moments} in terms of $N$, we rewrite the main term using the classical Selberg--Delange method. Theorem~\ref{thm:explicitck} is an explicit version of Theorem~\ref{theorem:usualhatmoments}.

\begin{theorem}\label{thm:explicitck}
For any $\epsilon>0$ and any positive integer $k$, we have
\[
 \frac{1}{\#\mathcal{D}(N)}\sum_{n\in\mathcal{D}(N)} (\#\Sel_{\hphi}(\hE_n))^k=c_k(\log N)^{\frac{1}{2}(3^k-1)}+O_{k,\epsilon}\left((\log N)^{\frac{3^{k}}{2}-\frac{5}{6}+\epsilon}\right),
\]
where
\begin{equation}\label{eq:ck}
c_k\coloneqq \frac{(2\cdot 3^{k+1}+1)(3^{k+1}+3^{-k+1}+7)}{7\cdot 13\cdot 
\Gamma(\frac{1}{2}(3^k+1))}\cdot 
\prod_{\substack{p\equiv 2\bmod 3\\p\neq 2}}
\frac{1+\frac{3^k}{p}+\frac{3^k}{p^{2}}}{1+\frac{1}{p}+\frac{1}{p^{2}}}
\prod_{p}\left(1-\frac{1}{p}\right)^{\frac{1}{2}(3^k-1)}.
\end{equation}
\end{theorem}
\begin{proof}
To rewrite the main term in~\eqref{eq:omega2moments}, we begin by spliting the sum according to $n\bmod 9$. 
First consider the case when $\eta\in\{\pm 1,\pm 2,\pm 4\}$, and use the orthogonality of Dirichlet characters to package this congruence condition 
\[\sum_{n\in\mathcal{D}_\eta(N)} 3^{k(\delta_\eta+\omega_2(\cf(2n)))}
=
\sum_{n\in\mathcal{D}(N)} \left(\frac{1}{\varphi(9)}\sum_{\chi\bmod 9}\overline{\chi}(\eta)\chi(n)\right)3^{k(\delta_\eta+\omega_2(\cf(2n))},
\]
where $\varphi$ denotes the Euler's totient function.
Now sum over $\eta\in\{\pm 1,\pm 2,\pm 4\}$ and rearrange the order of summation
\[\sum_{\eta\in\{\pm 1,\pm 2,\pm 4\}}\sum_{n\in\mathcal{D}_\eta(N)} 3^{k(\delta_\eta+\omega_2(\cf(2n)))}
=\frac{1}{\varphi(9)}
 \sum_{\chi\bmod 9}\left(\sum_{\eta}3^{k\delta_{\eta}}\overline{\chi}(\eta)\right)\sum_{n\in\mathcal{D}(N)}\chi(n)3^{k\omega_2(\cf(2n))}.
\]
For any non-principal character $\chi$, putting in the definition of $\delta_\eta$, we deduce that
\[\sum_{\eta}3^{k\delta_\eta}\overline{\chi}(\eta)=
\sum_{\eta}3^{k\delta_\eta}\overline{\chi}(\eta)-\sum_{\eta}\overline{\chi}(\eta)
=
\sum_{\eta}(3^{k\delta_\eta}-1)\overline{\chi}(\eta)
=
(3^{-k}-1)(\overline{\chi}(-4)+\overline{\chi}(4)),
\]
which is only nonzero for the two non-principal characters $\chi=\chi_1$ and $\overline{\chi}_1\bmod 9$ defined by $\chi_1(\pm 2)=\rho$.
Consider the following Dirichlet series and its Euler product
\begin{align*}
F_\chi(s)&=\sum_{n} \chi(n)3^{k\omega_2(\cf(2n))} n^{-s}\\
&=\left(3^k+\frac{\chi(2)3^{k}}{2^s}+\frac{\chi(2^2)}{2^{2s}}\right)\prod_{\substack{p\equiv 2\bmod 3\\p\neq 2}} \left(1+\frac{\chi(p)3^{k}}{p^s}+\frac{\chi(p^{2})3^{k}}{p^{2s}}\right)
\prod_{p\equiv 1\bmod 3} \left(1+\frac{\chi(p)}{p^s}+\frac{\chi(p^2)}{p^{2s}}\right).
\end{align*}
For the characters $\chi_1$ and $\overline{\chi}_1$, since $\chi_1(-4)+\chi_1(-1)+\chi_1(2)=\overline{\chi}_1(-4)+\overline{\chi}_1(-1)+\overline{\chi}_1(2)=0$, the pole of $F_{\chi_1}(s)$ and $F_{\overline{\chi}_1}(s)$ at $s=1$ must have order $<\frac{1}{2}(3^{k}+1)$. Therefore $\sum_{n} \chi_1(n)3^{k\omega_2(\cf(2n))}$ and $\sum_{n} \overline{\chi}_1(n)3^{k\omega_2(\cf(2n))}$ will be absorbed into an error term of $O\left(N(\log N)^{\frac{1}{2}(3^k-1)-1}\right)$.

For the principal character $\chi_0\bmod 9$, 
we compute
\[\sum_{\eta}3^{k\delta_\eta}\overline{\chi}_0(\eta)=
4+2\cdot 3^{-k}.
\] 
Moreover, we can see from the Euler product that $F_{\chi_0}$ has a pole at $s=1$ of order $\frac{1}{2}(3^k+1)$.
By a general version of the classical Selberg--Delange method by Tenenbaum~\cite[Theorem~II.5.2]{Tenenbaum}, we get the estimate
\[\sum_{n\in\mathcal{D}(N)}\chi_0(n)3^{k\omega_2(\cf(2n))} =
c_0N(\log N)^{\frac{1}{2}(3^k-1)}+O\left(N(\log N)^{\frac{1}{2}(3^k-1)-1}\right),\]
where 
\begin{align*}
c_0&
\coloneqq\frac{1}{\Gamma(\frac{1}{2}(3^k+1))}\cdot \lim_{s\rightarrow 1}\frac{F_{\chi_0}(s)}{\zeta(s)^{\frac{1}{2}(3^k+1)}}\\
&=\frac{1}{\Gamma(\frac{1}{2}(3^k+1))}
\left(3^k+\frac{3^{k}}{2}+\frac{1}{2^2}\right)\prod_{\substack{p\equiv 2\bmod 3\\p\neq 2}} \left(1+\frac{3^{k}}{p}+\frac{3^{k}}{p^2}\right)
\prod_{p\equiv 1\bmod 3} \left(1+\frac{1}{p}+\frac{1}{p^2}\right)\prod_{p}\left(1-\frac{1}{p}\right)^{\frac{1}{2}(3^k+1)}
\end{align*}

Similarly, for the case when $3\mid \eta$, we also obtain the estimate
\[\sum_{\eta\in\{0,\pm 3\}}\sum_{n\in\mathcal{D}_\eta(N)} 3^{k(\delta_\eta+\omega_2(\cf(2n)))}
 =\left(\frac{3^{k}}{3}+\frac{1}{3^2}\right)
c_0N(\log N)^{\frac{1}{2}(3^k-1)}+O\left(N(\log N)^{\frac{1}{2}(3^k-1)-1}\right).\]
Summing over all cases of $\eta$, we have
\begin{align*}
\sum_{n\in\mathcal{D}(N)} 3^{k(\delta_\eta+\omega_2(\cf(2n)))}
&=\left(\frac{1}{\varphi(9)}(4+2\cdot 3^{-k})+\left(\frac{3^{k}}{3}+\frac{1}{3^2}\right)\right)c_0
N(\log N)^{\frac{1}{2}(3^k-1)}
+O\left(N(\log N)^{\frac{1}{2}(3^k-1)-1}\right)
\\&
=\frac{1}{3}
\left(3^k+\frac{7}{3}+3^{-k}\right)\cdot c_0
N(\log N)^{\frac{1}{2}(3^k-1)}
+O\left(N(\log N)^{\frac{1}{2}(3^k-1)-1}\right).
\end{align*}
Putting this into~\eqref{eq:omega2moments}, and then combining with 
\[ \#\mathcal{D}(N)
=N\prod_p \left(1-\frac{1}{p^{3}}\right)+O\left(N^{\frac{1}{3}}\right)
=N\prod_p \left(1-\frac{1}{p}\right)\left(1+\frac{1}{p}+\frac{1}{p^{2}}\right)+O\left(N^{\frac{1}{3}}\right)
\]
yields the claim in~\eqref{eq:regularmoments}.
\end{proof}

\end{document}